\documentclass{article}
\usepackage[utf8]{inputenc}
\usepackage{amsthm}
\usepackage{amsmath}
\usepackage{amsfonts}
\usepackage{hyperref}
\usepackage{xcolor}
\hypersetup{
    colorlinks,
    linkcolor={red!80!black},
    citecolor={blue!80!black},
    urlcolor={blue!80!black}
}
\usepackage{float}
\usepackage[margin=1in]{geometry}
\usepackage{graphicx}
\usepackage{subcaption}
\usepackage{caption}
\usepackage{url}
\usepackage{tikz}

\DeclareMathOperator{\sech}{sech}
\DeclareMathOperator{\csch}{csch}

\newcommand{\ut}{\tilde{u}}
\newcommand{\vt}{\tilde{v}}
\newcommand{\wt}{\tilde{w}}

\newcommand{\ctau}{\alpha}
\newcommand{\ukdv}{\eta}

\newcommand{\tIcex}{\pmb{\tilde{c}}}
\newcommand{\tIAex}{\tilde{A}}
\newcommand{\tIbex}{\pmb{\tilde{b}}}
\newcommand{\tIcim}{\pmb{c}}
\newcommand{\tIAim}{A}
\newcommand{\tIbim}{\pmb{b}}

\newcommand{\tIIcex}{\pmb{\hat{\tilde{c}}}}
\newcommand{\tIIAex}{\hat{\tilde{A}}}
\newcommand{\tIIaex}{\pmb{\hat{\tilde{a}}}}
\newcommand{\tIIbexf}{\hat{\tilde{b}}_1}
\newcommand{\tIIbex}{\pmb{\hat{\tilde{b}}}}
\newcommand{\tIIcim}{\pmb{\hat{c}}}
\newcommand{\tIIAim}{\hat{A}}
\newcommand{\tIIaim}{\pmb{\hat{a}}}
\newcommand{\tIIbim}{\pmb{\hat{b}}}
\newcommand{\tIIbimf}{\hat{b}_1}

\newtheorem{thm}{Theorem}

\newtheorem{remark}{Remark}
\newtheorem{defn}{Definition}
\newtheorem{example}{Example}

\title{Traveling-wave solutions and structure-preserving numerical methods for
    a hyperbolic approximation of the Korteweg-de Vries equation}

\author{Abhijit Biswas\footnote{Corresponding author. Authors listed alphabetically.} \and David I. Ketcheson \and Hendrik Ranocha \and Jochen Sch\"utz}
\date{}

\begin{document}
\maketitle

\abstract{We study the recently-proposed hyperbolic approximation of the Korteweg-de Vries equation (KdV).
 We show that this approximation, which we call KdVH, possesses a rich variety of
 solutions, including solitary wave solutions that approximate KdV solitons, as well as other
 solitary and periodic solutions that are related to higher-order water wave models,
 and may include singularities.
 We analyze a class of implicit-explicit Runge-Kutta time discretizations for KdVH
 that are asymptotic preserving, energy conserving, and can be applied to other hyperbolized
 systems. We also develop structure-preserving spatial discretizations based on summation-by-parts
 operators in space including finite difference, discontinuous Galerkin, and Fourier methods. We use the
entropy relaxation approach to make the fully discrete schemes energy-preserving.
 Numerical experiments demonstrate the effectiveness of these discretizations.}

\section{Introduction}
The Korteweg-de Vries  equation (KdV)
\begin{align} \label{kdv}
    \partial_t \ukdv + \ukdv \partial_x \ukdv + \partial^3_x \ukdv = 0
\end{align}
is a widely studied model for water waves and perhaps the simplest known
nonlinear integrable partial differential equation (PDE) \cite{yang2010nonlinear}.
A hyperbolic approximation of KdV, referred to herein as KdVH,
has been proposed recently \cite{besse2022perfectly}:
\begin{subequations}\label{kdvH}
\begin{align}
\partial_t u + u\partial_x u + \partial_x w &= 0 \;, \label{kdvH-a} \\
\tau \partial_t v &= (\partial_x v - w) \;, \label{kdvH-b} \\
\tau \partial_t w &= -(\partial_x u - v) \;.\label{kdvH-c}
\end{align}
\end{subequations}
Here $\tau>0$ is referred to as the relaxation parameter.  When $\tau \to 0$
(referred to as the \emph{relaxation limit}),
formally one obtains $v \to u_x$ and $w \to u_{xx}$ so that \eqref{kdvH-a} becomes
equivalent to \eqref{kdv} with $u \to \ukdv$.

Similar hyperbolic approximations have been proposed for a number of other dispersive
nonlinear wave equations
\cite{antuono2009dispersive,grosso2010dispersive,mazaheri2016first,favrie2017rapid,
dhaouadi2019extended,escalante2019efficient,chesnokov2019hyperbolic,bassi2020hyperbolic,
gavrilyuk2022hyperbolic,chesnokov2023strongly}
as well as other classes of PDEs
\cite{toro2014advection,ruter2018hyperbolic,li2018new,dhaouadi2022hyperbolic,dhaouadi2022first,gavrilyuk2024conduit}.
A general approach to such approximations has been given in \cite{ketcheson2024approximation}.
The KdVH system \eqref{kdvH} was proposed with the idea of facilitating
the implementation of accurate nonreflecting boundary conditions in an approximation of
the KdV equation \cite{besse2022perfectly}.  In that work, the authors also briefly analyzed the
dispersion relation, hyperbolic structure, and one class of traveling wave solutions.
They presented numerical solutions obtained with 2nd-order operator splitting in
time and a 2nd-order MUSCL scheme in space.

In this work we further study the structure and numerical discretization of
the KdVH system \eqref{kdvH}.
In Section \ref{sec:traveling-waves} we perform a more complete study of traveling wave
solutions of KdVH.
The KdV equation admits soliton solutions of arbitrary large height and
speed.  Nevertheless, such waves beyond a certain height are known to be unphysical
\cite{brun2018convective}.  The KdVH system imposes a maximum speed
on such soliton-like waves and therefore avoids, at least qualitatively,
this non-physical aspect of the KdV model (see further discussion in Section \ref{sec:traveling-waves}).
We will also see that the KdVH system has traveling waves that are related to solutions
of the Camassa-Holm system \cite{camassa1994new}, in which solutions are known
to propagate at finite speed in a certain sense \cite{constantin2005finite}.

Given the many special properties of KdV \eqref{kdv}, it is natural to ask whether
or to what extent they are preserved by \eqref{kdvH}.  It is also desirable to
develop numerical discretizations of \eqref{kdvH} that preserve this structure
at the discrete level.  In Section \ref{sec:disc}, we develop high-order
accurate time and space discretizations that are asymptotic preserving and that
exactly preserve a quadratic invariant.
We present results on asymptotic preservation for Implicit-Explicit (ImEx) Runge-Kutta methods
under a range of assumptions.  In Section \ref{sec:experiments} we present
numerical examples that support the theoretical results and demonstrate the
effectiveness of this approach for approximating solutions of the KdV equation
\eqref{kdv}.

The new contributions of the present work include:
\begin{itemize}
    \item Characterization of the limits of soliton-like traveling wave solutions;
    \item Identification and analysis of other traveling wave solutions, including
            some with no counterpart in KdV but that seem to be related to other
            water wave models;
    \item Efficient ImEx time integration methods that are provably asymptotic preserving
            and asymptotically accurate;
    \item Energy-preserving full discretizations based on summation-by-parts operators in space and entropy 
    relaxation in time.
\end{itemize}

\section{Traveling wave solutions of KdVH}\label{sec:traveling-waves}
In this section we study traveling wave solutions of the KdVH system, focusing
first on waves that approximate KdV solitons in Section \ref{sec:soliton}
and then on a variety of other classes of traveling waves in Section \ref{sec:other-tw}.
The calculations in this section were performed using the Python packages
NumPy \cite{harris2020array}, SciPy \cite{virtanen2020scipy},
and Matplotlib \cite{hunter2007matplotlib}.

\subsection{KdV-soliton-like waves} \label{sec:soliton}
The soliton solutions of the KdV equation \eqref{kdv} take the form
\begin{align} \label{soliton}
    \ukdv(x,t) = 3c \sech^2\left(\frac{\sqrt{9c}(x-ct)}{6}\right).
\end{align}
The parameter $c$, which controls the width, amplitude, and speed of these
waves, can take any positive value.  Thus, the KdV equation possesses traveling
wave solutions with arbitrarily large velocity\footnote{In addition to the
soliton solutions, there also exist periodic traveling wave solutions of KdV
with any positive velocity.}.  This is perhaps not surprising, since the
phase velocity of small perturbations is also unbounded with respect to the wavenumber,
for the KdV equation.

Now we turn to the KdVH system.  To find traveling wave solutions, we apply the \emph{ansatz}
\begin{align} \label{tw-ansatz}
    u & = \ut(x-ct) & v & = \vt(x-ct) & w & = \wt(x-ct)
\end{align}
and we furthermore assume that each dependent variable tends
to a constant as $|\xi| := |x-ct|\to \infty$.
Then from \eqref{kdvH} we obtain the ordinary differential equation (ODE) system
\begin{subequations}
\begin{align}
    -c \ut' + \wt' + \ut \ut' & = 0 \label{Eq:kdvh_wave_ode_a} \\
    -c\tau \vt' & = \vt'-\wt \label{Eq:kdvh_wave_ode_b} \\
    -c \tau \wt' & = -\ut' + \vt. \label{Eq:kdvh_wave_ode_c}
\end{align}
\end{subequations}
We integrate the first equation and assume that $\ut, \wt$ tend to zero for large
$|x-ct|$ in order to determine the constant of integration.  We substitute
the result into the other two equations above to obtain the system
\begin{subequations}\label{Eq:traveling_wave_system}
\begin{align}
    \ut' & = \frac{1}{1+c\tau(\ut-c)} \vt \label{utilde-prime} \\
    \vt' & = \frac{1}{1+c\tau} (c-\ut/2) \ut. \label{vtilde-prime}
\end{align}
\end{subequations}
As depicted in Figure \ref{fig:phase_portrait_kdvh}, this system has two equilibrium points: $(\ut,\vt)=(0,0)$ and
$(\ut,\vt)=(2c,0)$.  For $\tau^{-1} > c^2$, the origin is a
hyperbolic point, while the other equilibrium is always a
center, with
\begin{align*}
 H(\ut, \vt) = \frac{\vt^2}{2} - \frac{\ut^2}{1+c\tau}\left(-\frac{c\tau} 8 \ut^2 + \frac{3c^2\tau-1}{6} \ut + \frac{c(1-c^2\tau)}2\right)
\end{align*}
as a first integral of the system. The system has a homoclinic connection (along the level set $H(\ut, \vt) = 0$), corresponding to a
solitary wave. Figure \ref{fig:phase_portrait_kdvh} shows the phase portrait of the system \eqref{Eq:traveling_wave_system},
along with a solitary wave solution obtained by numerically integrating \eqref{Eq:traveling_wave_system}.
This structure is essentially the same as that found when studying traveling wave solutions
of the KdV equation, and the homoclinic orbit tends to a KdV soliton as $\tau\to 0$.
Even though the relaxation parameter is not very small for the case plotted here, one can easily see that
for this solution $v\approx u_x$ and $w\approx v_x$.

\begin{figure}
    \center
    \includegraphics[width=4in]{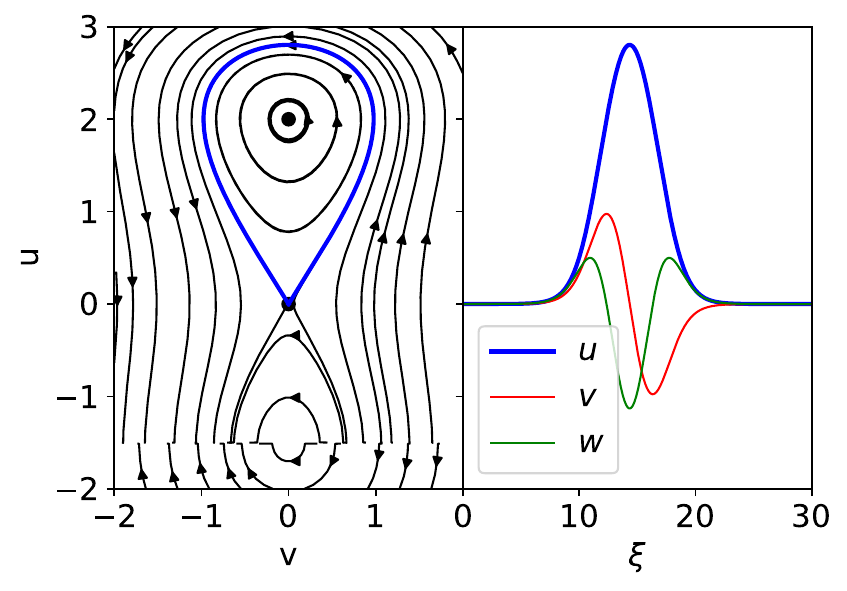}
    \caption{Phase portrait of the KdVH traveling wave system \eqref{Eq:traveling_wave_system} with wave speed $c=1$ and relaxation parameter $\tau = 2/5$.
    The homoclinic orbit, shown in blue, begins and ends at the saddle point $(0, 0)$, corresponding to a
    solitary traveling wave solution.  These waves tend to KdV solitons in the limit $\tau\to 0$.
   Note the presence of a line singularity, where the denominator in \eqref{utilde-prime} vanishes.}
    \label{fig:phase_portrait_kdvh}
\end{figure}

In order to compute these solitary waves more accurately, we use a Petviashvili-type algorithm
\cite{alvarez2014petviashvili}.  We rewrite \eqref{Eq:traveling_wave_system} as a single
second-order ODE for $\ut$.  Differentiating \eqref{utilde-prime} gives

\begin{align*}
    \vt' & = (1+c\tau(\ut-c)) \ut'' + c \tau (\ut')^2.
\end{align*}
Substituting this in \eqref{vtilde-prime}, we obtain after some simplification
\begin{align*}
    (1-c^2\tau)\ut'' + c\tau \ut \ut'' + c\tau (\ut')^2 & = \ut\frac{c-\ut/2}{1+c\tau}.
\end{align*}
This can be further rewritten as
\begin{align}\label{Eq:traveling_wave_kdvh_2nd}
- \ut'' + \frac{c}{(1 + c \tau)(1-c^2 \tau)}\ut = \frac{1}{(1 + c \tau)(1-c^2 \tau)} \frac{\ut^2}{2} + \frac{c\tau}{1 - c^2 \tau}\left(\ut \ut'\right)'.
\end{align}

Note that as $\tau \to 0$, \eqref{Eq:traveling_wave_kdvh_2nd} formally converges to
$$- \ut'' + c\ut = \frac{\ut^2}{2},$$
the equation satisfied by traveling waves of the KdV equation.
Next we write
\eqref{Eq:traveling_wave_kdvh_2nd} in the form $L \ut = N(\ut)$, where the linear operator is given by
$L = -D^2 + \frac{c}{(1 + c \tau)(1 - c^2 \tau)} I$, (with $D$ the derivative operator);
and the nonlinear operator is defined as
$N(\ut) = \frac{1}{(1 + c \tau)(1 - c^2 \tau)} \frac{\ut^2}{2} + \frac{c\tau}{1 - c^2 \tau}\left(\ut \ut'\right)'$.
Next we discretize, approximating the derivative operator $D$ using Fourier spectral differentiation.
Then we iteratively solve the system
$$L u^{[k+1]} = m(u^{[k]})^2 N(u^{[k]}), \quad k = 0,1,2,\ldots$$
where $u^{[0]}$ is an initial guess and $m(u^n) = \frac{\langle Lu^n, u^n \rangle}{\langle N(u^n), u^n \rangle}$ is
the stabilizing factor.
\begin{figure}[ht]
    \centering
    \includegraphics[width=4in]{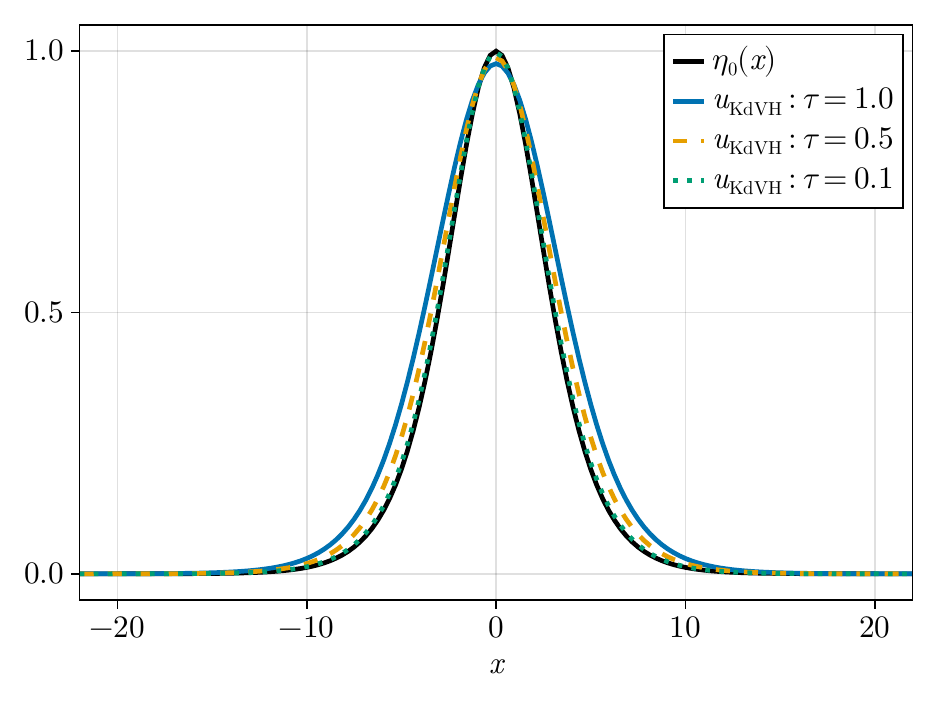}
    \caption{Comparison of the analytical soliton solution \eqref{soliton} (denoted by $\ukdv_0(x)$) of the KdV equation and the numerical solitary waves (denoted by
    $u_\mathrm{KdVH}: \tau$) for the KdVH system with different relaxation parameters $\tau$ and $c = \frac{1}{3}$.
    The solutions of the KdVH system converge to the exact solutions of the KdV equation
    as $\tau \to 0$, which is clearly evident.}
    \label{fig:solitary_wave_kdv_kdvh}
\end{figure}
Choosing the domain $[-30\pi, 30\pi]$ with $2^9$ spatial grid points, we compute the solitary waves of the KdVH system for
three different values of the relaxation parameter: $\tau = 1.0, 0.5, 0.1$ using the Petviashvili algorithm.
We iterate until the residual $\| L u^{[k]} - N(u^{[k]}) \|_\infty$ is approximately machine precision (for IEEE 64-bit floating point numbers).
Figure \ref{fig:solitary_wave_kdv_kdvh}
displays the resulting numerical KdVH solitary waves and the corresponding KdV soliton \eqref{soliton}.
We see that the KdVH solitary waves converge to the soliton as $\tau \to 0$.

The KdVH system also possesses periodic traveling waves with finite period (not shown here),
which tend to the well-known cnoidal solutions of the KdV equation as $\tau\to0$.

However, this does not fully describe the dynamics of the system \eqref{Eq:traveling_wave_system}.
Notably, for $c^2>\tau^{-1}$, the origin is not a saddle; instead the eigenvalues of the Jacobian
are purely imaginary.  Correspondingly, there is no homoclinic connection originating from the
origin.  Furthermore, as can be seen in the left plot of Figure
\ref{fig:phase_portrait_kdvh}, system \eqref{Eq:traveling_wave_system} becomes
singular when $\ut$ is chosen such that the denominator of \eqref{utilde-prime} vanishes.
These observations hint at the existence of other, quite different traveling waves, which
we study in the next section.

\subsection{Other traveling wave solutions} \label{sec:other-tw}
It turns out that the KdVH system admits a much larger set of traveling
wave solutions, some of which seem to be related to other dispersive water wave models.
To investigate general traveling waves, we note from \eqref{tw-ansatz} that such a wave with speed $c$
must satisfy
\begin{align*}
    v_t + c v_x & = 0 \\
    w_t + c w_x & = 0.
\end{align*}
These two equations stem from the representation of traveling waves, and
are as such independent of the actual form of the KdVH equation. Subsituting
$-cv_x$ and $-cw_x$, respectively, for $v_t$ and $w_t$, respectively, into
\eqref{kdvH-b} and \eqref{kdvH-c}, one obtains 
\begin{subequations}
\label{tw-constraints}
\begin{align}
    v & = u_x - \ctau w_x   \label{tw-constraints-a}\\
    w & = (1+\ctau) v_x,    \label{tw-constraints-b}
\end{align}
\end{subequations}
where $\ctau = c \tau$.
Substituting $v_x$ in \eqref{tw-constraints-b} through $u_{xx} - \alpha w_{xx}$, derived from \eqref{tw-constraints-a}, we end up with the equation
\begin{equation}
 w+\alpha(1+\alpha) w_{xx} \equiv \left(I+\alpha(1+\alpha)\partial_x^2\right)w = (1+\alpha) u_{xx}.
\end{equation}
Together with the requirement that $w$ is a smooth function that vanishes at $x = \pm \infty$, this constitutes a well-posed problem in $w$ and hence allows for a unique solution $w$ as a function of the right-hand side. The spatial derivative of said function is given by
\begin{equation}
 w_x = \left(I + \alpha(1+\alpha)\partial_x^2\right)^{-1} u_{xxx}.
\end{equation}
Substituting in \eqref{kdvH-a}, we obtain the dispersive equation
\begin{align*}
    u_t + u u_x + (1+\ctau) (I + \ctau (1+ \alpha)\partial_x^2)^{-1} u_{xxx} & = 0,
\end{align*}
or equivalently (upon multiplication by $(I + \alpha(1+\alpha)\partial_x^2)$)
\begin{align} \label{DP-like}
    u_t + u u_x + (1+\ctau) u_{xxx} + \ctau (1+\ctau) (u_{xxt} + 3u_x u_{xx} + u u_{xxx}) & = 0.
\end{align}
As $\ctau \to 0$ we formally recover the KdV equation, as expected, and if we apply the ansatz
\eqref{tw-ansatz} we find an ODE system equivalent to \eqref{Eq:traveling_wave_system}.
Indeed, this analysis is in a sense equivalent to that of \eqref{tw-ansatz}--\eqref{Eq:traveling_wave_system} above, but it
is revealing since \eqref{DP-like} involves the same set of terms
as the integrable Camassa-Holm and Degasperis-Procesi equations.  It is not possible to transform
\eqref{DP-like} exactly into either of those systems (see \cite[Dfns. 1--2]{constantin2009hydrodynamical}\footnote{However, \eqref{DP-like} does belong to the class of equations described in \cite[Prop. 2]{constantin2009hydrodynamical},
which approximate the Serre-Green-Naghdi equations.}),
but it is natural to expect that traveling wave solutions of KdVH may include waves
that are similar in nature to solutions of the Camassa-Holm or Degasperis-Procesi equations, and in fact this is the case.

Some examples of these traveling waves are shown in Figures \ref{fig:phase_portrait_kdvh_1}--\ref{fig:phase_portrait_kdvh_4}.
These include waves with arbitrarily large
positive speed, as well as left-going solitary waves that
have very nearly the shape of $\sqrt{\sech{x}}$, quite different from KdV solitons.
Furthermore, they include sharply-peaked solutions that are homoclinic connections
originating from a point on the line of singularity.
These solutions have no counterpart in solutions of KdV but seem to be related to solutions of
the Degasperis-Procesi or Camassa-Holm equations.

\begin{remark}
    For $c\tau>0$ (i.e., for right-going solutions of KdVH), the coefficient of $u_{xxt}$ in \eqref{DP-like}
    is positive.  This may seem problematic, since the linearization of \eqref{DP-like} is ill-posed in this
    case.  However, it should be remembered that not all solutions of \eqref{DP-like} are solutions of \eqref{kdvH};
    we additionally require that there exist a consistent solution of \eqref{tw-constraints}, which is equivalent to
    \begin{align*}
        v + \alpha(1+\alpha) v_{xx} & = u_x \\
        w + \alpha(1+\alpha) w_{xx} & = (1+\alpha) u_{xx}.
    \end{align*}
\end{remark}

\begin{figure}
    \centering
    \begin{subfigure}{0.41\textwidth}
        \centering
        \includegraphics[width=\textwidth]{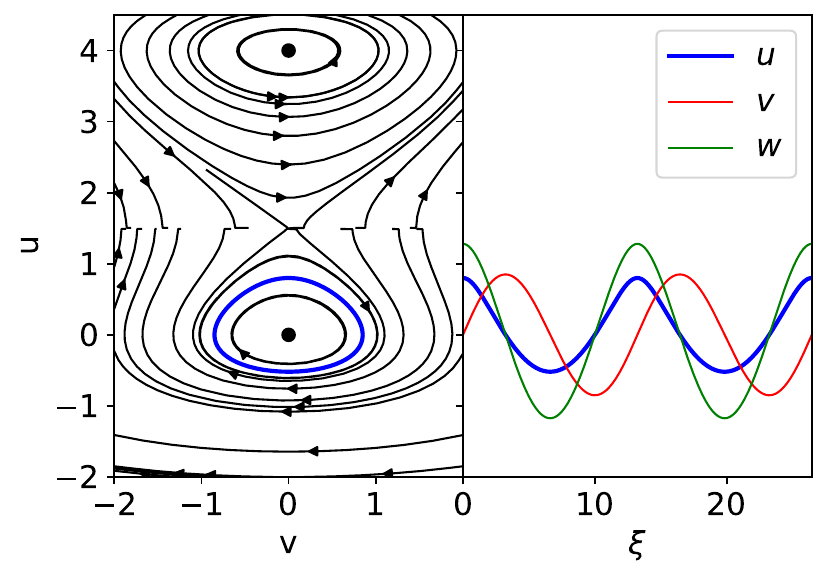}
        \caption{Periodic traveling wave: $\tau=1$, $c=2$}
        \label{fig:phase_portrait_kdvh_1}
    \end{subfigure}
    \hspace{0.05\textwidth}
    \begin{subfigure}{0.41\textwidth}
        \centering
        \includegraphics[width=\textwidth]{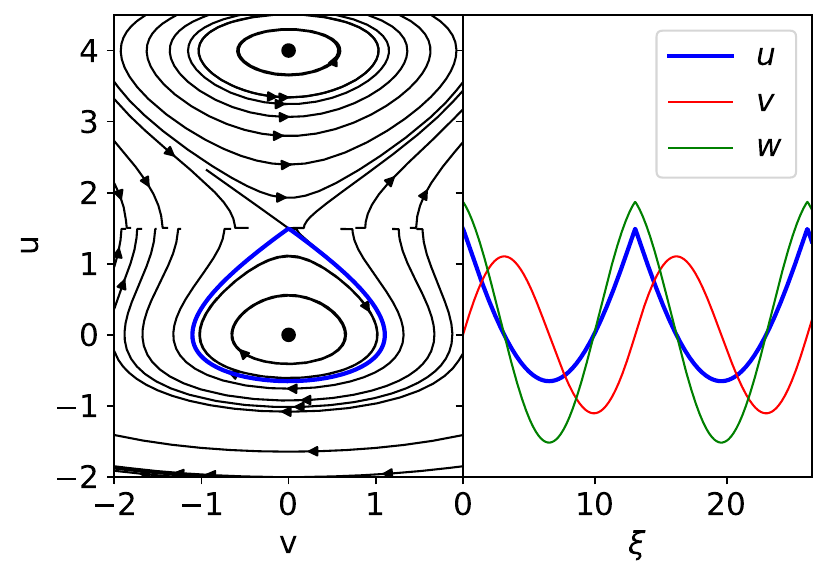}
        \caption{Peaked periodic traveling wave: $\tau=1$, $c=2$}
        \label{fig:phase_portrait_kdvh_2}
    \end{subfigure}
    \begin{subfigure}{0.41\textwidth}
        \centering
        \includegraphics[width=\textwidth]{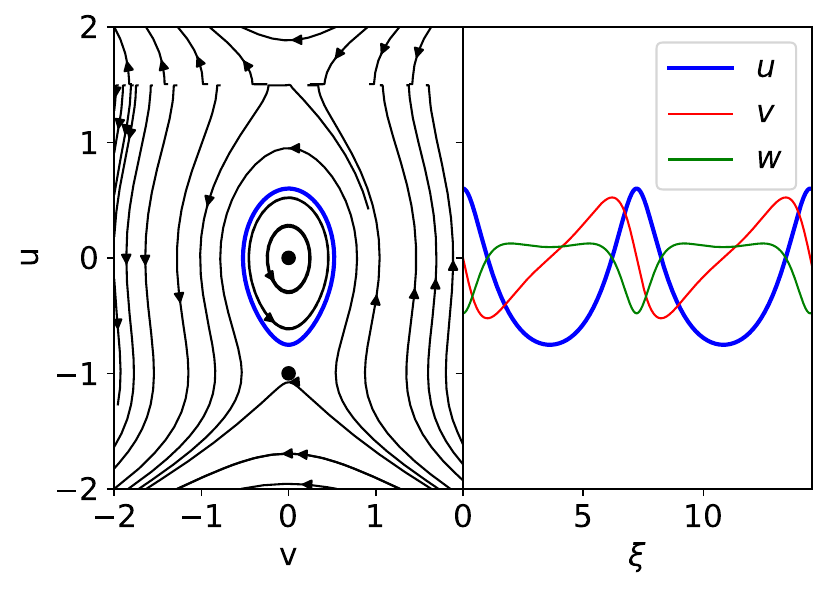}
        \caption{Periodic traveling wave: $\tau=1$, $c=-1/2$}
        \label{fig:phase_portrait_kdvh_3}
    \end{subfigure}
    \hspace{0.05\textwidth}
    \begin{subfigure}{0.41\textwidth}
        \centering
        \includegraphics[width=\textwidth]{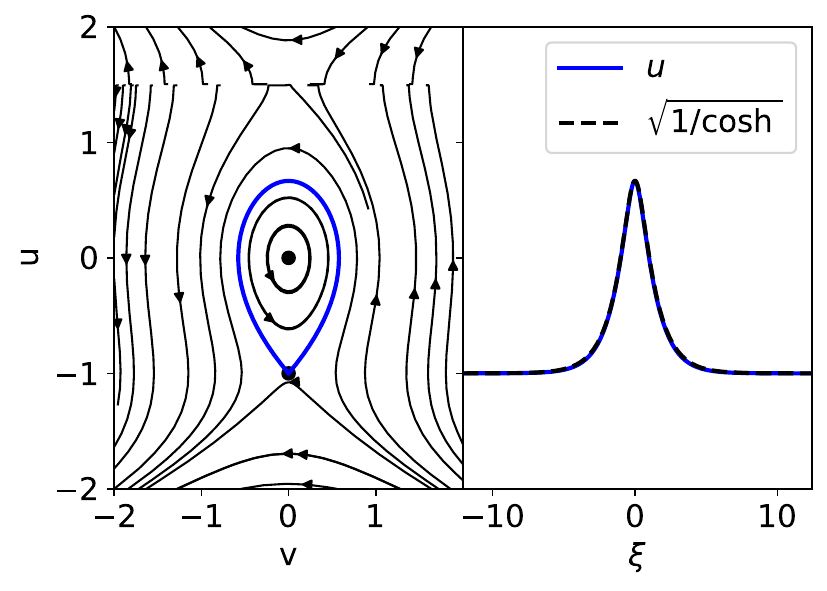}
        \caption{Solitary wave: $\tau=1$, $c=-1/2$}
        \label{fig:phase_portrait_kdvh_4}
    \end{subfigure}
    \caption{Traveling wave solutions of the KdVH equation.  These solutions have no counterpart in KdV solutions.
                Waves in the top row travel to the right at a speed greater than what is possible for soliton-like
                traveling waves.  Waves in the bottom row travel to the left.  The bottom-right solitary wave has
                nearly (but not quite) the shape $(5/3)\sqrt{\sech(5x/3)}-1$, which is plotted for comparison.
                Note the presence of a line singularity, where the denominator in \eqref{utilde-prime} vanishes.}
    \label{fig:traveling_waves}
\end{figure}

In experiments we have conducted (not shown here), these waves are not observed
to spontaneously emerge from general initial data, even if the initial data is
not well-prepared.  However, using any of these traveling wave solutions as
initial data, we observe that they are dynamically stable for long times.

\begin{remark}
The KdVH system can be written as a hyperbolic balance law
\begin{align}
    q_t + F'(q)q_x & = S(q).
\end{align}
The eigenvalues of the flux Jacobian $F'(q)$ evaluated at $u=0$ are $-\tau^{-1},
\pm\tau^{-1/2}$.  Notice that the greatest eigenvalue $+\sqrt{1/\tau}$ coincides precisely with the
maximum speed of the soliton-like traveling waves from the previous section.
\end{remark}

\section{Asymptotic-preserving and energy-conserving numerical discretization}\label{sec:disc}
In this section we develop structure-preserving discretizations for the KdVH system \eqref{kdvH},
focusing primarily on ImEx Runge-Kutta (RK) methods.
Let $q^n = [u^n, v^n, w^n]^T$ and $q^{n+1} = [u^{n+1}, v^{n+1}, w^{n+1}]^T$ denote the
numerical approximation of the solution at time $t_n$ and $t_{n+1} = t_n + \Delta t$, respectively.
Starting from the solution $q^n$ at time $t_n$, an ImEx-RK method applied to the system
\begin{align*}
    q_t = f(q) + g(q)
\end{align*}
computes the approximate solution at time $t_{n+1}$ as
\begin{subequations}
\label{eq:ImEx_mthd}
\begin{align}
    q^{(i)} & = q^n + \Delta t \left(\sum_{j = 1}^{i-1} \tilde{a}_{ij} f(q^{(j)}) + \sum_{j = 1}^{i} a_{ij} g(q^{(j)}) \right), \ i = 1,2, \ldots, s \;, \label{Eq:ImEx_mthd_stage_type_I}\\
    q^{n+1} & = q^n + \Delta t \left( \sum_{j = 1}^{s} \tilde{b}_{j} f(q^{(j)}) + \sum_{j = 1}^{s} b_{j} g(q^{(j)}) \right)  \; \label{Eq:ImEx_mthd_sol_type_I}.
\end{align}
\end{subequations}
This method can be compactly represented by the following Butcher tableau:
\begin{equation}\label{Mthd:ImEx-RK}
    \begin{array}{c|ccc}
      \pmb{\tIcex} & \tIAex \\
      \hline
      & \\[-1em]
      & \tIbex^T \\
    \end{array}
    \qquad
    \begin{array}{c|ccc}
      \tIcim & \tIAim \\
      \hline
      & \\[-1em]
      & \tIbim^T \\
    \end{array} \;,
  \end{equation}
where the matrix $\tIAex = (\tilde{a}_{ij}) \in \mathbb{R}^{s \times s}$ represents the explicit part, $A = (a_{ij}) \in \mathbb{R}^{s \times s}$ represents the implicit part, and the vectors $\tIcex$, $\tIbex$, $\tIcim$, and $\tIbim$ are in $\mathbb{R}^{s}$.
The choice of the splitting between the functions $f$ and $g$ is a key element
of the schemes introduced below.

\subsection{Asymptotic-preserving time discretization}\label{sec:AP-analysis}
The KdVH system \eqref{kdvH} includes both convective and algebraic terms that
depend on the relaxation parameter $\tau$ and become arbitrarily stiff as
$\tau\to 0$.  To deal with this stiffness we include the stiff terms in $g$,
to be integrated implicitly.
Our goal is to obtain accurate numerical solutions for all values of $\tau$.
In particular we seek a numerical discretization that tends to a consistent discretization of KdV as $\tau \to 0$;
this is known as an \emph{asymptotic-preserving (AP)} scheme \cite{boscarino2024asymptotic,jin2022asymptotic}.
In the rest of this section, we investigate the AP property for different classes of ImEx schemes.
In Section \ref{sec:experiments} we provide numerical tests that confirm the theoretical results.

The AP schemes provide an efficient methodology for solving multi-scale relaxation problems.
We denote the continuous problem \eqref{kdvH} with the parameter $\tau$ as $P^{\tau}$, and its limiting problem,
which reduces to the KdV equation, as $P^{0}$. For the development of AP schemes, it is sufficient to focus initially on time
discretization, leaving space continuous for now, with suitable spatial discretization to be adopted later. To this end, we
denote by $P^{\tau}_h$ the numerical discretization of the continuous problem $P^{\tau}$, where $h$ represents the
discretization parameters. For instance, $h = \Delta t$ denotes the time step in a problem that is continuous in space,
or $h = (\Delta t, \Delta x)$ in the fully discrete case. The AP property guarantees that for fixed discretization
parameters $h$, the scheme $P^{\tau}_h$ provides, in the limit $\tau \to 0$, a consistent discretization of the limit
problem $P^{0}$, denoted by $P^{0}_h$. This relationship is summarized in Figure \ref{fig:AP}.
Our primary focus in this section is on the limit $P^\tau_h \to P^0_h$ of the numerical discretizations,
although we will also comment on the convergence of the methods for $h \to 0$.
\begin{figure}[h!]
    \centering
    \begin{tikzpicture}
        \node (Pt) at (0, 0) {$P^\tau$};
        \node (P0) at (3, 0) {$P^0$};
        \node (Pttau) at (0, -3) {$P^\tau_h$};
        \node (P0tau) at (3, -3) {$P^0_h$};

        \draw[->] (Pt) -- (P0) node[midway, above] {$\tau \to 0$};
        \draw[->] (Pttau) -- (P0tau) node[midway, below] {$ \tau \to 0$};
        \draw[->] (Pttau) -- (Pt) node[midway, left] {$h \to 0$};
        \draw[->] (P0tau) -- (P0) node[midway, right] {$h \to 0$};
    \end{tikzpicture}
    \caption{A schematic illustration of the AP property \cite{Jin2012}.\label{fig:AP}}
\end{figure}
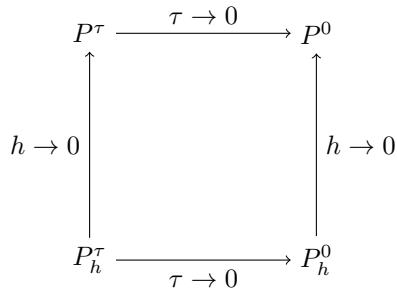

The AP property can be defined in general as follows:
\begin{defn}
    We say the numerical discretization $P^\tau_h$ is AP if
    the limiting discretization $P^0_h = \lim_{\tau \to 0}P^\tau_h$ is a
    consistent and stable discretization of the continuous limit model $P^0$.
\end{defn}
The AP property does not ensure that the scheme retains its order of accuracy
in the stiff limit $\tau \to 0$; for that we use the term \emph{asymptotically-accurate} (AA).
\begin{defn}
    The numerical discretization $P^\tau_h$ is AA if the
    local truncation error of $P^\tau_h$ with respect to $P^\tau$ is the same as that of $P^0_h$ with respect to $P^0$.
\end{defn}

The design of an AP ImEx scheme requires careful selection of the terms
to be integrated implicitly or explicitly.  It might seem most natural to
separate the hyperbolic terms from the algebraic (source) terms, but this turns
out not to provide the AP property.
Here we propose a different splitting that facilitates the construction of AP schemes.
Namely, we integrate explicitly only the nonlinear convective term, and integrate
implicitly all of the remaining (linear) terms:
\begin{equation}\label{Eq:KdVH_AP_split}
    \underbrace{\begin{pmatrix}
        u \\
        v \\
        w
    \end{pmatrix}_{t}}_{= q_t}
    =
    \underbrace{\begin{pmatrix}
        -u u_x \\
        0 \\
        0
    \end{pmatrix}}_{= f(q)}
    +
    \underbrace{\begin{pmatrix}
        -w_x \\
        \frac{1}{\tau} (v_x - w) \\
        \frac{1}{\tau} (-u_x + v)
    \end{pmatrix}}_{= g(q)}.
\end{equation}
Although the term $w_x$ may not seem to be stiff, recall that in the relaxation
limit it approximates the stiff third-derivative term. Also, including this term
in $g$ ensures that the system $q_t+g_x=0$ is strictly hyperbolic, which is an
important property with respect to stability \cite{NoeSch2014,RSIMEXFullEuler}.

Given this splitting, we will next establish the AP property for some ImEx RK methods.
To do so, we use the following specification of the AP property:
\begin{defn}\label{Def:AP-for-components}
    We say an ImEx RK method \eqref{eq:ImEx_mthd} applied to the splitting
    \eqref{Eq:KdVH_AP_split} of the KdVH system is AP for the $u$-component if
    $u^{n+1} \to \ukdv^{n+1}$ for $\tau \to 0$, where $\ukdv^{n+1}$ is the
    numerical solution of the KdV equation with splitting
    \begin{equation}\label{Eq:KdV_AP_split}
        \ukdv_t
        =
        \underbrace{-\ukdv \ukdv_x}_{= f(\ukdv)}
        \underbrace{-\ukdv_{xxx}}_{= g(\ukdv)}
    \end{equation}
    and the same ImEx RK method. We say it is AP for the auxiliary components $v$ and $w$ if
    $v^{n+1} \to \ukdv_x^{n+1}$ and $w^{n+1} \to \ukdv_{xx}^{n+1}$ for $\tau \to 0$.
\end{defn}

We now establish the AP property and provide sufficient conditions for the AA property for our proposed splitting of the KdVH
system using two classes of additive Runge-Kutta (ARK) methods --- type I (also known as type A) and type II (also known as type CK) \cite{boscarino2024asymptotic}.
An ImEx-RK method is classified as type I if the matrix $ \tIAim \in \mathbb{R}^{s \times s} $ is invertible.
An ARK method is of type II if $a_{11}=0$, in which case it can be written as
\begin{equation}\label{Mthd:CK_SA_ARK}
\begin{array}{c|ccc}
    0 & 0 \\
    \tIIcex & \tIIaex & \tIIAex\\
    \hline
    & \tIIbexf & \tIIbex^T
\end{array}
\qquad
\begin{array}{c|ccc}
    0 & 0 \\
    \tIIcim & \tIIaim & \tIIAim \\
    \hline
    & \tIIbimf & \tIIbim^T
\end{array} \;.
\end{equation}
Here, $\tIIcex$, $\tIIaex$, $\tIIbex$, $\tIIcim$, $\tIIaim$, and $\tIIbim$ are vectors in $\mathbb{R}^{s-1}$, $\tIIbexf$ and $\tIIbimf$ are scalars in $\mathbb{R}$, and $\tIIAex$, $\tIIAim$ are matrices in $\mathbb{R}^{(s-1) \times (s-1)}$, where $\tIIAim$ is assumed to be invertible and lower-triangular, so that the method is diagonally implicit.

This analysis, depending on the method class, involves two additional conditions \cite{boscarino2024asymptotic}.
 An ImEx-RK method is said to be \emph{globally stiffly accurate} (GSA) if
\begin{align}
\tilde{a}_{si} = \tilde{b}_{i} \quad \text{and} \quad a_{si} = b_{i}, \quad i = 1, 2, \dots, s.
\end{align}
Thus, the implicit part is stiffly accurate (SA) and the explicit part has a
first-same-as-last (FSAL) structure.

The initial data for KdVH \eqref{kdvH} is said to be \emph{well-prepared} if
$v(x,0) = Du$ and $w(x,0) = D^2u$.  Here $D=\partial_x$ in the present section,
where we consider solutions continuous in space.  For a fully discrete scheme,
$D$ should be replaced by the discrete derivative used in the scheme.

The proofs and results below are similar to those in the literature for
other classes of relaxation systems \cite{boscarino2024asymptotic}.
However, the relaxation system considered here, with the specific splitting
\eqref{Eq:KdVH_AP_split}, differs in such a way that the AP property cannot
be deduced directly from previous results.

We now present the following results regarding the AP and AA properties of type I methods:
\begin{thm}\label{Th:AP_result_type_I}
    An ImEx-RK method of type I applied to the splitting \eqref{Eq:KdVH_AP_split} of the hyperbolic approximation of the KdV
    equation is always AP for the $u$-component. For such a method, in the stiff limit $\tau \to 0$, we have
    \begin{align}
        u^{n+1} - \ukdv(t_{n+1}) = \mathcal{O}(\Delta t^p) \;,
    \end{align}
    where $p$ is the order of the ImEx-RK method.
    Furthermore, if the method is assumed to be globally stiffly accurate, it is also AP for the auxiliary components $v$ and $w$.
    In the stiff limit $\tau \to 0$, we have
    \begin{align}
        v^{n+1} - \ukdv_x(t_{n+1}) = \mathcal{O}(\Delta t^p) \quad \text{and} \quad w^{n+1} - \ukdv_{xx}(t_{n+1}) = \mathcal{O}(\Delta t^p)\;.
    \end{align}
\end{thm}
\begin{proof}
Let us denote by $\pmb{u} = [u^{(1)}, u^{(2)}, \dots, u^{(s)}]^T$, $\pmb{v} = [v^{(1)}, v^{(2)}, \dots, v^{(s)}]^T$, and $\pmb{w} = [w^{(1)}, w^{(2)}, \dots, w^{(s)}]^T$ the vectors of stage-solution components for the variables $u$, $v$, and $w$, respectively, and let $\pmb{e} = [1, 1, \dots, 1]^T$ be the vector of ones in $\mathbb{R}^s$. For the KdVH system with the splitting \eqref{Eq:KdVH_AP_split}, the equation \eqref{Eq:ImEx_mthd_stage_type_I} in component form becomes
\begin{subequations}
\begin{align}
    \pmb{u} &= u^n \pmb{e} + \Delta t \left( \tIAex (-\pmb{u} \pmb{u}_x) + \tIAim(-\pmb{w}_x) \right) \;, \label{Eq:stage_u_type_I} \\
    \pmb{v} &= v^n \pmb{e} + \frac{\Delta t}{\tau}  \tIAim(\pmb{v}_x-\pmb{w}) \;, \label{Eq:stage_v_type_I} \\
    \pmb{w} &= w^n \pmb{e} + \frac{\Delta t}{\tau}  \tIAim(-\pmb{u}_x+\pmb{v}) \;. \label{Eq:stage_w_type_I}
\end{align}
\end{subequations}
Similarly, the final update of the solution can be written in components as
\begin{subequations}
\begin{align}
    u^{n+1} &= u^n  + \Delta t \left(\tIbex^T  (-\pmb{u} \pmb{u}_x) + \tIbim^T  (-\pmb{w}_x)  \right) \;, \label{Eq:sol_u_type_I} \\
    v^{n+1} &= v^n  + \frac{\Delta t}{\tau}  \tIbim^T  (\pmb{v}_x-\pmb{w}) \;, \label{Eq:sol_v_type_I} \\
    w^{n+1} &= w^n + \frac{\Delta t}{\tau}  \tIbim^T  (-\pmb{u}_x+\pmb{v})  \;. \label{Eq:sol_w_type_I}
\end{align}
\end{subequations}
We assume there exist Hilbert expansions for $u^n$, $v^n$, and $w^n$:
\begin{align} \label{hilbert}
u^n = u^n_0 + \tau u^n_1 + \tau^2 u^n_2 + \cdots, \quad v^n  = v^n_0 + \tau v^n_1 + \tau^2 v^n_2 + \cdots, \quad w^n = w^n_0 + \tau w^n_1 + \tau^2 w^n_2 + \cdots \;,
\end{align}
and for the stage vectors $\pmb{u}$, $\pmb{v}$, and $\pmb{w}$ of the stage solution components:
\begin{align*}
\pmb{u} = \pmb{u}_0 + \tau \pmb{u}_1 + \tau^2 \pmb{u}_2 + \cdots, \quad \pmb{v}  = \pmb{v}_0 + \tau \pmb{v}_1 + \tau^2 \pmb{v}_2 + \cdots, \quad \pmb{w} = \pmb{w}_0 + \tau \pmb{w}_1 + \tau^2 \pmb{w}_2 + \cdots \;.
\end{align*}

We insert these expansions into the stage equations and compare the leading-order terms in the powers of $\tau$. The leading-order terms in the expansions for the stage equations \eqref{Eq:stage_v_type_I} and \eqref{Eq:stage_w_type_I} yield:
\begin{align*}
    \mathcal{O}\left(\tau^{-1}\right) &\colon
    \begin{cases}
        \tIAim \left((\pmb{v}_0)_x - \pmb{w}_0\right) = \pmb{0} \;, \\
         \tIAim  \left(-(\pmb{u}_0)_x + \pmb{v}_0\right) = \pmb{0} \;.
    \end{cases}
\end{align*}
Since $\tIAim$ is invertible, it follows that
\begin{align}\label{Eq:Aux_var_rel_type_I}
   \pmb{v}_0 =  (\pmb{u}_0)_x \quad \text{and} \quad  \pmb{w}_0 =  (\pmb{v}_0)_x \;.
\end{align}

The leading-order term in the expansion of \eqref{Eq:stage_u_type_I} gives
\begin{align*}
    \mathcal{O}\left(\tau^{0}\right) & \colon \pmb{u}_0 = u_0^n \pmb{e} + \Delta t \left( - 	\tIAex \pmb{u}_0 (\pmb{u}_0)_x  - \tIAim(\pmb{w}_0)_x   \right)\;.
\end{align*}
Using $\pmb{v}_0 =  (\pmb{u}_0)_x$ and $\pmb{w}_0 =  (\pmb{v}_0)_x$ in the previous equation, we obtain
\begin{align}\label{Eq:Asymp_stage_u_type_I}
    \pmb{u}_0 = u_0^n \pmb{e} + \Delta t \left(- \tIAex \pmb{u}_0 (\pmb{u}_0)_x  - \tIAim(\pmb{u}_0)_{xxx} \right)\;.
\end{align}
Using the equation \eqref{Eq:Aux_var_rel_type_I} in the leading-order term of the solution update for the first equation yields:
\begin{align}\label{Eq:Asymp_sol_u_type_I}
    u_0^{n+1} &= u_0^n  + \Delta t \left( - \tIbex^T  \pmb{u}_0 (\pmb{u}_0)_x  -\tIbim^T  (\pmb{u}_0)_{xxx}  \right) \;.
\end{align}
In the limit as $\tau \to 0$ the numerical scheme becomes:
\begin{subequations}
    \begin{align*}
        \pmb{u}_0 & = u_0^n \pmb{e} + \Delta t \left(- \tIAex \pmb{u}_0 (\pmb{u}_0)_x  - \tIAim(\pmb{u}_0)_{xxx} \right)\;, \\
        u_0^{n+1} &= u_0^n  + \Delta t \left( - \tIbex^T  \pmb{u}_0 (\pmb{u}_0)_x  -\tIbim^T  (\pmb{u}_0)_{xxx}  \right) \;.
    \end{align*}
\end{subequations}
This is precisely the numerical scheme obtained from the time-stepping method \eqref{Mthd:ImEx-RK} of type I when applied to
the KdV equation $\ukdv_t = - \ukdv \ukdv_x -\ukdv_{xxx} $, where the term $-\ukdv \ukdv_x$ is treated explicitly and the term $-\ukdv_{xxx}$ is treated
implicitly. Thus, we also have $u^{n+1} - \eta(t_{n+1}) = \mathcal{O}(\Delta t^p)$, where $p$ is the order of the ImEx-RK method.

In addition if the method is GSA then we can prove the AP property for the auxiliary variables $v$ and $w$. To prove this, we use the equations \eqref{Eq:stage_v_type_I} and \eqref{Eq:stage_w_type_I}. By utilizing the invertibility of $\tIAim$, we obtain the following:
\begin{subequations}
\begin{align*}
    \frac{\Delta t}{\tau} (\pmb{v}_x-\pmb{w}) & = \tIAim^{-1}(\pmb{v}-v^n \pmb{e}) \;,  \\
    \frac{\Delta t}{\tau} (-\pmb{u}_x+\pmb{v}) & = \tIAim^{-1}(\pmb{w}-w^n \pmb{e})  \;.
\end{align*}
\end{subequations}
Inserting these expressions into the update rules for the auxiliary components, \eqref{Eq:sol_v_type_I} and \eqref{Eq:sol_w_type_I}, we obtain:
\begin{subequations}
\begin{align*}
    v^{n+1} &= v^n  + \tIbim^T\tIAim^{-1}(\pmb{v}-v^n \pmb{e}) =  \tIbim^T\tIAim^{-1}\pmb{v} + (1-\tIbim^T\tIAim^{-1} \pmb{e} )v^n  \;, \\
    w^{n+1} & = w^n + \tIbim^T\tIAim^{-1}(\pmb{w}-w^n \pmb{e}) =  \tIbim^T\tIAim^{-1}\pmb{w} + (1-\tIbim^T\tIAim^{-1} \pmb{e} )w^n  \;.
\end{align*}
\end{subequations}

Since the method is assumed to be GSA, the stiff accuracy of $\tIAim$ implies that $\tIbim^T\tIAim^{-1} = [0, 0, \dots, 1]$ in $\mathbb{R}^{s}$, hence $1 - \tIbim^T\tIAim^{-1} \pmb{e} = 0$. Using this information, we obtain:
\begin{subequations}
\begin{align*}
    v^{n+1} &= \tIbim^T\tIAim^{-1}\pmb{v} = v^{(s)} \;, \\
    w^{n+1} &= \tIbim^T\tIAim^{-1}\pmb{w} = w^{(s)} \;.
\end{align*}
\end{subequations}

In the limit as $\tau \to 0$, we have $v_0^{n+1} = v_0^{(s)}$ and $w_0^{n+1} = w_0^{(s)}$. Using equation
\eqref{Eq:Aux_var_rel_type_I}, we can express $v_0^{n+1} = (u_0^{(s)})_x$ and $w_0^{n+1} = (u_0^{(s)})_{xx}$.
By the GSA property of the scheme, it follows from equation \eqref{Eq:Asymp_stage_u_type_I} and \eqref{Eq:Asymp_sol_u_type_I}
that $u_0^{(s)} = u_0^{n+1}$, which leads to $v_0^{n+1} = (u_0^{n+1})_x$ and $w_0^{n+1} = (u_0^{n+1})_{xx}$.
This establishes the AP property of the type I ImEx-RK method for the auxiliary components. In the stiff limit $\tau \to 0$,
the following error estimates hold for the $v$ and $w$ components:
$$v^{n+1} - \ukdv_{x}(t_{n+1}) = \mathcal{O}(\Delta t^p) \ \text{and} \ w^{n+1} - \ukdv_{xx}(t_{n+1}) = \mathcal{O}(\Delta t^p) \;.$$
\end{proof}

\begin{remark} \label{rem:no-GSA}
    Without the GSA assumption in Theorem~\ref{Th:AP_result_type_I}, the AP property is obtained only for the $u$-component.
    In this case, the auxiliary variables may not remain on the manifold of equilibria, which could lead to a
    degradation in their accuracy. However, this does not affect the order of accuracy of the $u$-component.
\end{remark}

We now consider ImEx-RK methods of type II and prove the AP property for these methods. Before establishing the results for a general method in this class, we first examine the asymptotic-preserving property of the simple type II method ARS(1,1,1) when applied to the splitting \eqref{Eq:KdVH_AP_split}. The ARS(1,1,1) method updates the solution from time step $t^n$ to $t^{n+1}$ as follows:
$$
q^{n+1} = q^n + \Delta t f\left(q^{n}\right) + \Delta t g\left(q^{n+1}\right) \;,
$$
where the components of the update are given by:
\begin{subequations}
\begin{align}
u^{n+1} & = u^{n} + \Delta t \left(- u^{n} u_x^{n} -w_x^{n+1} \right)  \;, \label{Eq:ARS_AP_a}  \\
v^{n+1} & = v^{n} + \frac{\Delta t }{\tau}\left(v_x^{n+1} - w^{n+1}  \right)   \;, \label{Eq:ARS_AP_b}  \\
w^{n+1} & = w^{n} + \frac{\Delta t }{\tau}\left(-u_x^{n+1} + v^{n+1}  \right)   \;. \label{Eq:ARS_AP_c}
\end{align}
\end{subequations}
We now insert the Hilbert expansions \eqref{hilbert}
into the updating equations and analyze the leading-order term in the expansions in terms of $\tau$. The leading-order terms in expansions of equations \eqref{Eq:ARS_AP_b} and \eqref{Eq:ARS_AP_c} yield
\begin{align*}
    \mathcal{O}\left(\tau^{-1}\right) &\colon w^{n+1}_0  = (v_0^{n+1})_x  \quad \text{and} \quad v^{n+1}_0  = (u_0^{n+1})_x \;.
\end{align*}
The equation \eqref{Eq:ARS_AP_a} implies
\begin{align*}
    \mathcal{O}\left(\tau^{0}\right) &\colon u^{n+1}_0 = u^{n}_0 - \Delta t u^{n}_0 (u^{n}_0)_x  -\Delta t  (w^{n+1}_0)_x\;.
\end{align*}
 Using $w^{n+1}_0  = (v_0^{n+1})_x$ and $v^{n+1}_0  = (u_0^{n+1})_x$ in the previous equation, we obtain
\begin{align*}
     u^{n+1}_0 & = u^{n}_0 - \Delta t u^{n}_0 (u_0^{n})_x -\Delta t (u^{n+1}_0)_{xxx}  \;, \\
     \frac{u^{n+1}_0-u^{n}_0}{\Delta t } + u^{n}_0 (u_0^{n})_x  + (u^{n+1}_0)_{xxx}  & = 0 \;,
\end{align*}
which, in the limit $\tau \to 0$, is the discretization of the original KdV equation by the ARS(1,1,1) method. This proves the AP property of the ARS(1,1,1) method for the splitting \eqref{Eq:KdVH_AP_split}. To prove the AP property for a general ImEx-RK method of type II, we require an additional assumption of well-preparedness of the initial data. The well-preparedness of the initial data for our system, following \cite{boscarino2024asymptotic}, is given by
\begin{equation}\label{Eq:well-prepared_IC}
    v^{0} = u^{0}_x +  \mathcal{O}(\tau) \quad \text{and} \quad w^{0} = u^{0}_{xx} +  \mathcal{O}(\tau).
\end{equation}
For a general ImEx-RK method of type II, we prove the following result:
\begin{thm}\label{Th:AP_result_type_II}
    A globally stiffly accurate ImEx-RK method of type II, applied to the splitting \eqref{Eq:KdVH_AP_split} of the hyperbolic
    approximation of the KdV equation, together with the well-prepared initial data \eqref{Eq:well-prepared_IC}, is AP for all
    components $u$, $v$, and $w$. Furthermore, in the stiff limit $\tau \to 0$, the following error estimates apply to all
    components:
    \begin{align*}
        u^{n+1} - \ukdv(t_{n+1}) = \mathcal{O}(\Delta t^p), \ v^{n+1} - \ukdv_{x}(t_{n+1}) = \mathcal{O}(\Delta t^p),
        \textrm{and} \ w^{n+1} - \ukdv_{xx}(t_{n+1}) = \mathcal{O}(\Delta t^p)\;,
    \end{align*}
    where $p$ is the order of the ImEx-RK method.
\end{thm}
\begin{proof}
Let $q^n = [u^n, v^n, w^n]^T$ and $q^{n+1} = [u^{n+1}, v^{n+1}, w^{n+1}]^T$ denote the numerical approximations to the true solution at time $t_n$ and $t_{n+1} = t_n + \Delta t$, respectively.
Starting from the solution $q^n$ at time $t_n$, an ImEx-RK method of type II \eqref{Mthd:CK_SA_ARK} applied to the system
\begin{align*}
    q_t = f(q) + g(q)
\end{align*}
computes the approximate solution at time $t_{n+1}$ as
\begin{subequations}\label{Eq:ImEx_mthd}
\begin{align}
    q^{(1)} & = q^n \;, \\
    q^{(i)} & = q^n + \Delta t \left( \hat{\tilde{a}}_{i1} f(q^n) + \sum_{j = 2}^{i-1} \hat{\tilde{a}}_{ij} f(q^{(j)}) + \hat{a}_{i1} g(q^n) + \sum_{j = 2}^{i} \hat{a}_{ij} g(q^{(j)}) \right), \ i = 2,3, \ldots, s \;, \\
    q^{n+1} & = q^n + \Delta t \left( \tIIbexf f(q^n) + \sum_{j = 2}^{s} \hat{\tilde{b}}_{j} f(q^{(j)}) + \tIIbimf g(q^n) + \sum_{j = 2}^{s} \hat{b}_{j}  g(q^{(j)}) \right) \;.
\end{align}
\end{subequations}
Using the notation $\pmb{q} = [q^{(2)}, q^{(3)}, \dots, q^{(s)}]^T$, the identity matrix $I$ in $\mathbb{R}^{3 \times 3}$, and the vector of ones $\pmb{e}$ in $\mathbb{R}^{s-1}$, this can be written compactly as
\begin{subequations}
\begin{align}
    \pmb{q} &= \pmb{e} \otimes q^n + \Delta t \left( \tIIaex \otimes f(q^n) + (\tIIAex \otimes I) f(\pmb{q}) + \tIIaim \otimes g(q^n) + (\tIIAim \otimes I) g(\pmb{q}) \right) \;, \label{Eq:ImEx_stage} \\
    q^{n+1} &= q^n + \Delta t \left( \tIIbexf f(q^n) + (\tIIbex^T \otimes I) f(\pmb{q}) + \tIIbimf g(q^n) + (\tIIbim^T \otimes I) g(\pmb{q}) \right) \;. \label{Eq:ImEx_sol}
\end{align}
\end{subequations}
Let us denote by $\pmb{u} = [u^{(2)}, u^{(3)}, \dots, u^{(s)}]^T$, $\pmb{v} = [v^{(2)}, v^{(3)}, \dots, v^{(s)}]^T$, and $\pmb{w} = [w^{(2)}, w^{(3)}, \dots, w^{(s)}]^T$ the vectors of stage-solution components for the variables $u$, $v$, and $w$, respectively. For the KdVH system with the splitting \eqref{Eq:KdVH_AP_split}, the equation \eqref{Eq:ImEx_stage} in component form becomes
\begin{subequations}
\begin{align}
    \pmb{u} &= u^n \pmb{e} + \Delta t \left( (-u^n u_{x}^n)\tIIaex + \tIIAex (-\pmb{u} \pmb{u}_x)  -w_x^n \tIIaim + \tIIAim(-\pmb{w}_x) \right) \;, \label{Eq:stage_u} \\
    \pmb{v} &= v^n \pmb{e} + \frac{\Delta t}{\tau} \left( (v_x^n - w^n) \tIIaim + \tIIAim(\pmb{v}_x-\pmb{w})\right) \;, \label{Eq:stage_v} \\
    \pmb{w} &= w^n \pmb{e} + \frac{\Delta t}{\tau} \left( (-u_x^n + v^n) \tIIaim + \tIIAim(-\pmb{u}_x+\pmb{v})\right) \;. \label{Eq:stage_w}
\end{align}
\end{subequations}
Similarly, the final update of the solution can be written in components as
\begin{subequations}
\begin{align}
    u^{n+1} &= u^n  + \Delta t \left( \tIIbexf(-u^n u_{x}^n) +  \tIIbex^T  (-\pmb{u} \pmb{u}_x) + \tIIbimf(-w_x^n)  + \tIIbim^T  (-\pmb{w}_x)  \right) \;, \label{Eq:sol_u} \\
    v^{n+1} &= v^n  + \frac{\Delta t}{\tau} \left( \tIIbimf(v_x^n - w^n)  + \tIIbim^T  (\pmb{v}_x-\pmb{w}) \right) \;, \label{Eq:sol_v} \\
    w^{n+1} &= w^n + \frac{\Delta t}{\tau} \left( \tIIbimf(-u_x^n + v^n)  + \tIIbim^T  (-\pmb{u}_x+\pmb{v}) \right)  \;. \label{Eq:sol_w}
\end{align}
\end{subequations}
We assume there exist Hilbert expansions \eqref{hilbert} for $u^n$, $v^n$, and $w^n$
and similar expansions for the stage vectors $\pmb{u}$, $\pmb{v}$, and $\pmb{w}$ of the stage solution components:
\begin{align*}
\pmb{u} = \pmb{u}_0 + \tau \pmb{u}_1 + \tau^2 \pmb{u}_2 + \cdots, \quad \pmb{v}  = \pmb{v}_0 + \tau \pmb{v}_1 + \tau^2 \pmb{v}_2 + \cdots, \quad \pmb{w} = \pmb{w}_0 + \tau \pmb{w}_1 + \tau^2 \pmb{w}_2 + \cdots \;.
\end{align*}

We insert these expansions into the stage equations and compare the leading-order terms in the powers of $\tau$.
The leading-order terms in the expansions for the stage equations \eqref{Eq:stage_v} and \eqref{Eq:stage_w} yield:
\begin{align*}
    \mathcal{O}\left(\tau^{-1}\right) &\colon
    \begin{cases}
        \left((v_0^n)_x - w_0^n \right) \tIIaim + \tIIAim  \left((\pmb{v}_0)_x - \pmb{w}_0\right) = \pmb{0} \;, \\
        \left(-(u_0^n)_x + v_0^n \right) \tIIaim + \tIIAim  \left(-(\pmb{u}_0)_x + \pmb{v}_0\right) = \pmb{0} \;.
    \end{cases}
\end{align*}
The well-prepared initial data at time $t_n$ implies that $v_0^n = (u_0^n)_x$ and $w_0^n = (v_0^n)_x$.
Since $\tIIAim$ is invertible, it follows that
\begin{align}\label{Eq:Aux_var_rel}
   \pmb{v}_0 = (\pmb{u}_0)_x \quad \text{and} \quad \pmb{w}_0 = (\pmb{v}_0)_x \;.
\end{align}

The leading-order term in the expansion of \eqref{Eq:stage_u} gives
\begin{align*}
    \mathcal{O}\left(\tau^{0}\right) & \colon \pmb{u}_0 = u_0^n \pmb{e} + \Delta t \left( - u_0^n (u_{0}^n)_x \tIIaex - \tIIAex \pmb{u}_0 (\pmb{u}_0)_x -(w_0^n)_x \tIIaim - \tIIAim(\pmb{w}_0)_x \right)\;.
\end{align*}
Using $\pmb{v}_0 =  (\pmb{u}_0)_x$, $\pmb{w}_0 =  (\pmb{v}_0)_x$, and the well-preparedness of the initial data at time $t_n$
in the previous equation, we obtain
\begin{align}\label{Eq:Asymp_stage_u}
    \pmb{u}_0 & = u_0^n \pmb{e} + \Delta t \left(  - u_0^n (u_{0}^n)_x \tIIaex - \tIIAex \pmb{u}_0 (\pmb{u}_0)_x -(u_0^n)_{xxx} \tIIaim - \tIIAim(\pmb{u}_0)_{xxx}  \right)\;.
\end{align}

The solution update for the first equation in the leading-order term becomes:
\begin{align}\label{Eq:Asymp_sol_u}
    u_0^{n+1} &= u_0^n  + \Delta t \left(  - \tIIbexf u_0^n (u_{0}^n)_x  - \tIIbex^T  \pmb{u}_0 (\pmb{u}_0)_x -\tIIbimf(u_0^n)_{xxx}  - \tIIbim^T  (\pmb{u}_0)_{xxx} \right) \;.
\end{align}
In the limit as $\tau \to 0$, we have $u^n = u^n_0$ and $\pmb{u} = \pmb{u}_0$, so the numerical scheme becomes:
\begin{subequations}
    \begin{align*}
        \pmb{u}_0 & = u_0^n \pmb{e} + \Delta t \left(  - u_0^n (u_{0}^n)_x \tIIaex - \tIIAex \pmb{u}_0 (\pmb{u}_0)_x -(u_0^n)_{xxx} \tIIaim - \tIIAim(\pmb{u}_0)_{xxx}  \right)\;, \\
        u_0^{n+1} &= u_0^n  + \Delta t \left(  - \tIIbexf u_0^n (u_{0}^n)_x  - \tIIbex^T  \pmb{u}_0 (\pmb{u}_0)_x -\tIIbimf(u_0^n)_{xxx}  - \tIIbim^T  (\pmb{u}_0)_{xxx} \right)\;.
    \end{align*}
\end{subequations}
This is precisely the numerical scheme obtained from the time-stepping method \eqref{Mthd:CK_SA_ARK} when applied to the KdV
equation $\ukdv_t = - \ukdv \ukdv_x -\ukdv_{xxx} $, where the term $-\ukdv \ukdv_x$ is treated explicitly and the term $-\ukdv_{xxx}$ is treated implicitly.
It is important to note that, in order to preserve this result at the subsequent time step, the auxiliary components must
remain well-prepared. Therefore, we must project $v^{n+1}$ and $w^{n+1}$ onto the corresponding equilibrium manifolds. To achieve
this, we utilize the GSA property of the method. Since the method is assumed to satisfy the GSA property, it follows that:
\begin{subequations}
\begin{align*}
    v^{n+1} &= \tIIbim^T\tIIAim^{-1}\pmb{v} = v^{(s)} \;, \\
    w^{n+1} &= \tIIbim^T\tIIAim^{-1}\pmb{w} = w^{(s)} \;.
\end{align*}
\end{subequations}

In the limit as $\tau \to 0$, we have $v_0^{n+1} = v_0^{(s)}$ and $w_0^{n+1} = w_0^{(s)}$. Using equation
\eqref{Eq:Aux_var_rel}, we can express $v_0^{n+1} = (u_0^{(s)})_x$ and $w_0^{n+1} = (u_0^{(s)})_{xx}$.
By the GSA property of the scheme, it follows from equations \eqref{Eq:Asymp_stage_u} and \eqref{Eq:Asymp_sol_u}
that $u_0^{(s)} = u_0^{n+1}$, which leads to $v_0^{n+1} = (u_0^{n+1})_x$ and $w_0^{n+1} = (u_0^{n+1})_{xx}$.
This establishes the AP property, and the error estimates for all components follow trivially, thereby proving
the AA property of the method.
\end{proof}

\begin{remark} \label{rem:req-GSA}
    Note that the well-preparedness of the initial data is assumed at time $t_n$. This is justified by the fact that if
    the initial condition is well-prepared, as given by \eqref{Eq:well-prepared_IC}, then Theorem \ref{Th:AP_result_type_II}
    for $n = 0$ ensures that the solution is well-prepared for the step $n = 1$. Hence, by an induction argument,
    the auxiliary components will remain on the local equilibrium for all subsequent times. Additionally,
    note that the GSA property is necessary to guarantee that the auxiliary variables lie on the local equilibrium manifolds,
    which ensures the accuracy of the $u$-components and, consequently, the auxiliary components.
\end{remark}

\subsection{Energy-conserving spatial semidiscretization} \label{sec:space-disc}

As a completely integrable nonlinear PDE, the KdV equation possesses
an infinite set of conserved quantities, can be derived from a
Lagrangian or Hamiltonian formalism, and possesses soliton solutions
that are related to certain symmetries.  Here we review these properties
and related properties of the KdVH system.

Consider the Lagrangian
\begin{align*}
    L = -\frac{\phi_t \phi_x}{2} - \frac{(\phi_x)^3}{6} - \frac{\phi_x\phi_{xxx}}{2} \;.
\end{align*}
The corresponding Euler-Lagrange equation is
\begin{align*}
    \phi_{xt} + \phi_x\phi_{xx} + \phi_{xxxx} = 0 \;.
\end{align*}
Introducing the variable $ \ukdv = \phi_x $, we obtain the KdV equation \eqref{kdv}.

As noted in the appendix of \cite{besse2022perfectly}, the KdVH system
\eqref{kdvH} can also be derived from a variational principle, starting
from the augmented Lagrangian
\begin{align}\label{Aug_L_KdVH}
    \hat{L} = -\frac{\phi_t \phi_x}{2} - \frac{\phi_x^3}{6} - \phi_x \chi_x - \tau \frac{\chi_t \chi_x}{2} - \tau \frac{\psi_t \psi_x}{2} + \frac{\psi_x^2}{2} + \psi \chi_x \;.
\end{align}
The corresponding Euler-Lagrange equations are
\begin{subequations}
\begin{align*}
    \phi_{xt} + \phi_x\phi_{xx} + \chi_{xx} &= 0 \;, \\
    \tau \psi_{xt} - \psi_{xx} + \chi_{x} &= 0 \;, \\
    \tau \chi_{xt} + \phi_{xx} - \psi_{x} &= 0 \;.
\end{align*}
\end{subequations}
Setting $u=\phi_x$, $v=\psi_x$, and $w=\chi_x$, we recover the hyperbolized system \eqref{kdvH}.

Similarly, the KdVH system \eqref{kdvH} can be expressed as a Hamiltonian PDE; the Hamiltonian is
obtained from the augmented Lagrangian \eqref{Aug_L_KdVH} using the Legendre
transformation. This yields the Hamiltonian density:
\begin{align*}
    \mathcal{H}(\phi, \psi, \chi, \phi_x, \psi_x, \chi_x) &= \frac{\partial \hat{L}}{\partial \phi_t} \phi_t + \frac{\partial \hat{L}}{\partial \psi_t} \psi_t + \frac{\partial \hat{L}}{\partial \chi_t} \chi_t - \hat{L} \\
    &= \frac{\phi_x^3}{6} + \phi_x \chi_x - \psi \chi_x - \frac{\psi_x^2}{2} \;.
\end{align*}
In terms of $ u $, $ v $, and $ w $, this is
\begin{align}
    H(u, v, w) = \int_{x_L}^{x_R} \left(\frac{u^3}{6} + uw - wF(v) - \frac{v^2}{2}\right) dx \;,
\end{align}
where
\begin{align*}
    F(v)(x,t) = \int_{x_L}^{x} v(y,t) dy \;.
\end{align*}
Then the KdVH system can be expressed in the form $q_t = \mathcal{J} \delta H(q)$ where
$q=[u,v,w]^T$ and
\begin{align}
    \delta H(q) & =
    \begin{bmatrix}
        \frac{u^2}{2} + w \\
        -v + F(w) \\
        u - F(v)
    \end{bmatrix} \; &
    \mathcal{J} & =
    \begin{bmatrix}
        -\partial_x & &  \\
        & -\frac{\partial_x}{\tau} & \\
        &  & -\frac{\partial_x}{\tau}
    \end{bmatrix} \;,
\end{align}
the KdV hyperbolized system \eqref{kdvH} can be expressed in the Hamiltonian PDE form $ q_t = \mathcal{J} \delta H(q) $.

The first three conserved quantities of the KdV equation \eqref{kdv} are
\begin{subequations} \label{KdV_invariants}
    \begin{align}
        \int \ukdv \, dx \quad (\text{mass}), \\
        \int \frac{1}{2} \ukdv^2 \, dx \quad (\text{energy}), \\
        \int \left(2\ukdv^3 - \ukdv_x^2\right) dx \quad (\text{Whitham}).
    \end{align}
\end{subequations}
For the KdVH system, it can be verified that the mass is conserved but the
other nonlinear invariants are not exactly conserved.  However, the system
conserves a modified energy:
\begin{align}\label{Mod_energy_KdVH}
    I(u,v,w) = \int_{x_L}^{x_R} \left(\frac{u^2}{2} + \tau \frac{v^2}{2} + \tau \frac{w^2}{2}\right) dx \;,
\end{align}
which satisfies $ -q_x = \mathcal{J} \delta I(q) $ and hence we obtain the relative equilibrium structure of the Hamiltonian PDE \cite{duran1998numerical,duran2000numerical}.
Thus, this modified energy is a conserved quantity of the KdVH system \eqref{kdvH}.
It can also be easily verified that the modified energy given in \eqref{Mod_energy_KdVH} is a conserved quantity of the system by checking that
\begin{align*}
\frac{dI}{dt} & = \int_{x_L}^{x_R} \left(uu_t + \tau vv_t + \tau ww_t\right) dx \\
              & = \int_{x_L}^{x_R} \left(u (-uu_x - w_x) + v (v_x-w) + w (v-u_x)\right) dx \\
              & = \int_{x_L}^{x_R} \left(-\left(\frac{u^3}{3}\right)_x-(uw)_x+\left(\frac{v^2}{2}\right)_x\right) dx \\
              & = 0 \;.
\end{align*}
Formally, the modified energy \eqref{Mod_energy_KdVH} converges to the KdV energy
for $\tau \to 0$.

Conservation of the modified energy \eqref{Mod_energy_KdVH} of KdVH and
the original energy \eqref{KdV_invariants} of KdV can be shown using
integration by parts and a split form of the nonlinear term, similar to
Burgers' equation \cite[eq.~(6.40)]{richtmyer1967difference}. Alternatively,
the chain rule can be used. However, the chain rule does not have a discrete
equivalent in general \cite{ranocha2019mimetic}. Thus, we use the first
approach and mimic integration by parts discretely using
summation-by-parts (SBP) operators \cite{svard2014review,fernandez2014review}.
In this article, we only need periodic SBP operators; see \cite{ranocha2021broad}
and references cited therein for several examples and details.

\begin{defn}
    A \emph{periodic first-derivative SBP operator}
    consists of
    a grid $\pmb{x}$,
    a consistent first-derivative operator $D$,
    and a symmetric and positive-definite matrix $M$
    such that
    \begin{equation}
    \label{eq:D1-periodic}
      M D + D^T M = 0.
    \end{equation}
\end{defn}
The operator $D$ is skew-symmetric with respect to the mass matrix $M$.
In other words, the product $M D$ is skew-symmetric in the usual sense. In particular,
$\pmb{y}^T M D \pmb{y} = 0$ for any vector $\pmb{y}$. We will use this property
several times to analyze energy conservation and refer to it as the
SBP property \eqref{eq:D1-periodic}.

We will also use upwind operators following \cite{mattsson2017diagonal}.
\begin{defn}
    A \emph{periodic first-derivative upwind SBP operator}
    consists of
    a grid $\pmb{x}$,
    consistent first-derivative operators $D_\pm$,
    and a symmetric and positive-definite matrix $M$
    such that
    \begin{equation}
    \label{eq:D1-upwind-periodic}
      M D_+ + D_-^T M = 0,
      \quad
      \frac{1}{2} M (D_+ - D_-) \text{ is negative semidefinite}.
    \end{equation}
\end{defn}

For upwind SBP operators, the arithmetic average $D = (D_+ + D_-) / 2$ is
an SBP operator \cite{mattsson2017diagonal}.
The semidiscretizations using SBP operators use a collocation approach,
i.e., all nonlinear operations are performed pointwise. For example,
$\pmb{u} = \bigl(u(\pmb{x}_i)\bigr)_{i=1}^N$
is the vector of the values of the function $u$ at the grid points and
$\pmb{u}^2 = \bigl(u(\pmb{x}_i)^2\bigr)_{i=1}^N$. We will only use
diagonal-norm operators, i.e., operators with diagonal mass/norm matrix $M$.

\begin{example}
    First-order accurate periodic first-derivative finite difference
    SBP operators are given by
    \begin{equation}
    \begin{aligned}
        D_+ &= \frac{1}{\Delta x} \begin{pmatrix}
            -1 & 1 & 0 & \cdots & 0 \\
            0 & -1 & 1 & \cdots & 0 \\
            \vdots & \vdots & \vdots & \ddots & \vdots \\
            0 & 0 & \cdots & -1 & 1 \\
            1 & 0 & \cdots & 0 & -1
        \end{pmatrix},
        &
        D_- &= \frac{1}{\Delta x} \begin{pmatrix}
            1 & 0 & \cdots & 0 & -1 \\
            -1 & 1 & 0 & \cdots & 0 \\
            0 & -1 & 1 & \cdots & 0 \\
            \vdots & \vdots & \vdots & \ddots & \vdots \\
            0 & 0 & \cdots & -1 & 1
        \end{pmatrix},
        \\
        D = \frac{D_+ + D_-}{2} &= \frac{1}{2 \Delta x} \begin{pmatrix}
            0 & 1 & 0 & \cdots & -1 \\
            -1 & 0 & 1 & \cdots & 0 \\
            0 & -1 & 0 & \cdots & 1 \\
            \vdots & \vdots & \vdots & \ddots & \vdots \\
            1 & 0 & \cdots & -1 & 0
        \end{pmatrix},
        &
        M &= \Delta x \operatorname{I},
    \end{aligned}
    \end{equation}
    where $\Delta x$ is the grid spacing and $\operatorname{I}$ is the
    identity matrix. The central SBP operator $D$ is actually second-order
    accurate. Moreover,
    \begin{equation}
        D_+ D D_-
        =
        \frac{1}{\Delta x^3} \begin{pmatrix}
            0 & -1 & 1/2 & 0 & \dots & 0 & -1/2 & 1 \\
            1 & 0 & -1 & 1/2 & \dots & 0 & 0 & -1/2 \\
            -1/2 & 1 & 0 & -1 & \dots & 0 & 0 & 0 \\
            0 & -1/2 & 1 & 0 & \dots & 0 & 0 & 0 \\
            \vdots & \vdots & \vdots & \vdots & \ddots & \vdots & \vdots & \vdots \\
            0 & 0 & 0 & 0 & \dots & 0 & -1 & 1/2 \\
            1/2 & 0 & 0 & 0 & \dots & 1 & 0 & -1 \\
            -1 & 1/2 & 0 & 0 & \dots & -1/2 & 1 & 0
        \end{pmatrix}
    \end{equation}
    is a second-order accurate approximation of the third derivative.
\end{example}

Using the split-form discretization of the nonlinear term together with
an upwind discretization of the third derivative leads to the semidiscretization
\begin{equation}
\label{eq:KdV-semidiscrete}
    \partial_t \pmb{\ukdv}
    + \frac{1}{3} \left( D \pmb{\ukdv}^2 + \pmb{\ukdv} D \pmb{\ukdv} \right)
    + D_+ D D_- \pmb{\ukdv}
    =
    \pmb{0}
\end{equation}
of the KDV equation \eqref{kdv}.

\begin{thm}
\label{thm:KdV-semidiscrete}
    The semidiscretization \eqref{eq:KdV-semidiscrete} conserves the discrete
    counterparts
    \begin{equation}
        \pmb{1}^T M \pmb{\ukdv} \approx \int \ukdv,
        \qquad
        \frac{1}{2} \pmb{1}^T M \pmb{\ukdv}^2 \approx \frac{1}{2} \int \ukdv^2
    \end{equation}
    of the linear and quadratic invariant \eqref{KdV_invariants} of the KdV
    equation \eqref{kdv} if periodic diagonal-norm upwind SBP operators with the
    same mass matrix $M$ are used.
\end{thm}
\begin{proof}
    The linear invariant is conserved since
    \begin{equation}
    \begin{aligned}
        \partial_t \pmb{1}^T M \pmb{\ukdv}
        &=
        \pmb{1}^T M \partial_t \pmb{\ukdv}
        =
        - \frac{1}{3} \pmb{1}^T M D \pmb{\ukdv}^2
        - \frac{1}{3} \pmb{1}^T M \pmb{\ukdv} D \pmb{\ukdv}
        - \pmb{1}^T M D_+ D D_- \pmb{\ukdv}
        \\
        &=
        - \frac{1}{3} \pmb{1}^T M D \pmb{\ukdv}^2
        - \frac{1}{3} \pmb{\ukdv}^T M D \pmb{\ukdv}
        - \pmb{1}^T M D_+ D D_- \pmb{\ukdv}
        \\
        &=
        + \frac{1}{3} \pmb{1}^T D^T M \pmb{\ukdv}^2
        - \frac{1}{6} \pmb{\ukdv}^T M D \pmb{\ukdv}
        + \frac{1}{6} \pmb{\ukdv}^T D^T M \pmb{\ukdv}
        + \pmb{1}^T D_-^T M D D_- \pmb{\ukdv}
        =
        0.
    \end{aligned}
    \end{equation}
    In the second line, we have used that the mass matrix $M$ is diagonal.
    The SBP properties are used in the third line.
    Please note that we have used the SBP property \eqref{eq:D1-periodic}
    for half of the
    term $- \frac{1}{3} \pmb{\ukdv}^T M D \pmb{\ukdv}$.
    The final step follows from consistency of the derivative operators
    ($D \pmb{1} = \pmb{0}$) and the symmetry of $M$.
    Finally, we obtain
    \begin{equation}
        \partial_t \pmb{1}^T M \pmb{\ukdv}^2
        =
        \partial_t \bigl( \pmb{\ukdv}^T M \pmb{\ukdv} \bigr)
        =
        2 \pmb{\ukdv}^T M \partial_t \pmb{\ukdv},
    \end{equation}
    where we have used that $M$ is diagonal. Then, we have
    \begin{equation}
    \begin{aligned}
        \frac{1}{2} \partial_t \pmb{1}^T M \pmb{\ukdv}^2
        &=
        - \frac{1}{3} \pmb{\ukdv}^T M D \pmb{\ukdv}^2
        - \frac{1}{3} \pmb{\ukdv}^T M \pmb{\ukdv} D \pmb{\ukdv}
        - \pmb{\ukdv}^T M D_+ D D_- \pmb{\ukdv}
        \\
        &=
        - \frac{1}{3} \pmb{\ukdv}^T M D \pmb{\ukdv}^2
        - \frac{1}{3} (\pmb{\ukdv}^2)^T M D \pmb{\ukdv}
        + \pmb{\ukdv}^T D_-^T M D D_- \pmb{\ukdv}
        =
        0,
    \end{aligned}
    \end{equation}
    where we have again used the SBP properties \eqref{eq:D1-periodic}
    and \eqref{eq:D1-upwind-periodic} in the last line and
    the fact that $\pmb{y}^T M D \pmb{y} = 0$, i.e., that $M D$ is
    skew-symmetric in the usual sense.
\end{proof}

Next, we consider the semidiscretization
\begin{equation}
\label{eq:KdVH-semidiscrete}
\begin{aligned}
    \partial_t \pmb{u}
    + \frac{1}{3} \left( D \pmb{u}^2 + \pmb{u} D \pmb{u} \right) + D_+ \pmb{w} &= \pmb{0},
    \\
    \partial_t \pmb{v}
    + \frac{1}{\tau} \left( -D \pmb{v} + \pmb{w} \right) &= \pmb{0},
    \\
    \partial_t \pmb{w}
    + \frac{1}{\tau} \left( D_- \pmb{u} - \pmb{v} \right) &= \pmb{0}
\end{aligned}
\end{equation}
of the KdVH system \eqref{kdvH}.

\begin{thm}
\label{thm:KdVH-semidiscrete}
    The semidiscretization \eqref{eq:KdVH-semidiscrete} conserves the discrete
    counterparts
    \begin{equation}
        \pmb{1}^T M \pmb{u} \approx \int u,
        \qquad
        \frac{1}{2} \pmb{1}^T M \left( \pmb{u}^2 + \tau \pmb{v}^2 + \tau \pmb{w}^2 \right) \approx \frac{1}{2} \int \left( u^2 + \tau v^2 + \tau w^2 \right)
    \end{equation}
    of the linear and quadratic invariant \eqref{Mod_energy_KdVH} of the KdVH
    system \eqref{kdvH} if periodic diagonal-norm upwind SBP operators with the
    same mass matrix $M$ are used.
\end{thm}
\begin{proof}
    Conservation of the linear invariant follows as in the proof of
    Theorem~\ref{thm:KdV-semidiscrete}. The quadratic invariant is conserved
    since
    \begin{equation}
    \begin{aligned}
        &\quad
        \frac{1}{2} \partial_t \pmb{1}^T M \left( \pmb{u}^2 + \tau \pmb{v}^2 + \tau \pmb{w}^2 \right)
        =
        \pmb{u}^T M \partial_t \pmb{u}
        + \tau \pmb{v}^T M \partial_t \pmb{v}
        + \tau \pmb{w}^T M \partial_t \pmb{w}
        \\
        &=
        - \frac{1}{3} \pmb{u}^T M D \pmb{u}^2
        - \frac{1}{3} \pmb{u}^T M \pmb{u} D \pmb{u}
        - \pmb{u}^T M D_+ \pmb{w}
        + \pmb{v}^T M D \pmb{v}
        - \pmb{v}^T M \pmb{w}
        - \pmb{w}^T M D_- \pmb{u}
        + \pmb{w}^T M \pmb{v}
        \\
        &=
        - \frac{1}{3} \pmb{u}^T M D \pmb{u}^2
        - \frac{1}{3} (\pmb{u}^2)^T M D \pmb{u}
        - \pmb{u}^T M D_+ \pmb{w}
        + \pmb{v}^T M D \pmb{v}
        + \pmb{w}^T D_+^T M \pmb{u}
        =
        0,
    \end{aligned}
    \end{equation}
    where we have used that $M$ is diagonal and the SBP properties,
    in particular the skew-symmetry of $M D$.
\end{proof}

Formally taking the limit $\tau \to 0$, the semidiscretization
\eqref{eq:KdVH-semidiscrete} of the KdVH system \eqref{kdvH} converges
to the semidiscretization \eqref{eq:KdV-semidiscrete} of the KdV equation
\eqref{kdv}. Moreover, the quadratic invariant of the KdVH semidiscretization
converges to the quadratic invariant of the KdV semidiscretization. We can use this
to prove fully-discrete counterparts of Theorems~\ref{Th:AP_result_type_I} and
\ref{Th:AP_result_type_II}. Thus, we introduce the splitting
\begin{equation}
\label{eq:KdVH-semidiscrete-split}
    \underbrace{\partial_t \begin{pmatrix}
        \pmb{u} \\
        \pmb{v} \\
        \pmb{w}
    \end{pmatrix}}_{\partial_t \pmb{q}}
    =
    \underbrace{\begin{pmatrix}
        -\frac{1}{3} \left( D \pmb{u}^2 + \pmb{u} D \pmb{u} \right) \\
        \pmb{0} \\
        \pmb{0}
    \end{pmatrix}}_{f(\pmb{q})}
    +
    \underbrace{\begin{pmatrix}
        -D_+ \pmb{w} \\
        \frac{1}{\tau} (D \pmb{v} - \pmb{w}) \\
        \frac{1}{\tau} (-D_- \pmb{u} + \pmb{v})
    \end{pmatrix}}_{g(\pmb{q})}.
\end{equation}

\begin{thm}\label{Th:AP_result_type_I_fully_discrete}
    An ImEx-RK method of type I applied to the splitting \eqref{eq:KdVH-semidiscrete-split}
    of the semidiscrete hyperbolic approximation of the KdV equation is always AP for the
    $\pmb{u}$-component. Furthermore, if the method is assumed to be globally stiffly accurate,
    then it is also AP for the auxiliary components $\pmb{v}$ and $\pmb{w}$.
\end{thm}

\begin{thm}\label{Th:AP_result_type_II_fully_discrete}
    A globally stiffly accurate ImEx-RK method of type II, applied to the splitting
    \eqref{eq:KdVH-semidiscrete-split} of the semidiscrete hyperbolic approximation of
    the KdV equation, along with well-prepared initial data, is asymptotic-preserving
    for all the components $\pmb{u}$, $\pmb{v}$, and $\pmb{w}$.
\end{thm}

The proofs of Theorems~\ref{Th:AP_result_type_I_fully_discrete} and
\ref{Th:AP_result_type_II_fully_discrete} are analogous to the proofs of
Theorems~\ref{Th:AP_result_type_I} and \ref{Th:AP_result_type_II}. Thus,
we omit them here.

\subsection{Fully-discrete energy conservation via RK relaxation}
To extend energy conservation from the semidiscrete to the fully-discrete level,
we utilize the entropy relaxation Runge-Kutta technique \cite{ketcheson2019relaxation,ranocha2020relaxation,ranocha2020general}.
This approach modifies standard RK methods
to conserve a single invariant for ODE systems. Applications of this entropy relaxation approach in ImEx time integration for
conserving single or multiple invariants have been explored in
\cite{biswas2023accurate,biswas2023multiple,li2022implicit}.
Here, we focus on conserving a single invariant at the fully-discrete level using ImEx RK methods.
The energy-conserving spatial semidiscretization of the KdVH or the original KdV equation reduces the problem to an
ODE system for $q \in \mathbb{R}^{m}$, given by
\begin{align*}
  q_t = f(q) + g(q), \quad q(0) = q^0,
\end{align*}
with the energy invariant denoted by $I(q)$.
Using an ImEx method \eqref{eq:ImEx_mthd}, we require $I(q^{n+1}) = I(q^n) = I(q^0)$
at the discrete level.
However, standard methods generally do not satisfy $I(q^{n+1}) = I(q^n)$.
To address this, the entropy relaxation approach introduces
a scalar entropy relaxation parameter, $\gamma_n$, modifying the solution update as
\begin{align*}
    q(t_n + \gamma_n \Delta t)
    \approx
    q^{n+1}_{\gamma_n}
    =
    q^n + \gamma_n \Delta t \left( q^{n+1} - q^{n} \right),
\end{align*}
with $\gamma_n$ chosen to satisfy the nonlinear scalar equation
\begin{align*}
I(q^{n+1}_{\gamma_n}) = I(q^n).
\end{align*}
Under certain mild conditions, $\gamma_n = 1 + \mathcal{O}(\Delta t^{p-1})$
exists and can be determined at each time step for a general nonlinear
invariant $I$, where $p$ is the order of the ImEx method
\cite{ketcheson2019relaxation,ranocha2020general}.
For specific invariants, such as energy represented by a solution norm, explicit
formulas for $\gamma_n$ are available \cite{ketcheson2019relaxation,ranocha2020general}.
For an ImEx RK method with Butcher coefficients given in \eqref{Mthd:ImEx-RK}, the explicit formula for $\gamma_n$ is given by
\begin{align}
    \gamma_n & = \frac{-2 \  \langle q^n, \ \Delta t \sum_{j=1}^{s} (\tilde{b}_j f_j + b_j g_j)   \rangle}{\langle\Delta t \sum_{j=1}^{s} (\tilde{b}_j f_j + b_j g_j), \ \Delta t \sum_{j=1}^{s} (\tilde{b}_j f_j + b_j g_j)  \rangle} \;,
\end{align}
where $f_j = f(q^{(j)})$ and $g_j = g(q^{(j)})$, are the function evaluations of $f$ and $g$ at the $j$th stage solution $q^{(j)}$.
By constructing the entropy relaxation parameter,
the modified solution preserves the invariant of the semidiscretized system, thereby ensuring the preservation of
the invariant in the fully-discrete numerical scheme.

In the numerical experiments below, this entropy relaxation technique is used only
where indicated in Section \ref{sec:test-conservation} and not at all in Section \ref{sec:test-AP}.

\section{Numerical experiments} \label{sec:experiments}

In this section, we present the numerical results that validate the effectiveness and accuracy
of the proposed methods for solving the KdVH system.
We study solutions involving one or two solitons;
a step-like initial condition leading to a dispersive
shock wave has already been studied in \cite{besse2022perfectly}.
We have implemented the numerical methods in Julia \cite{bezanson2017julia}.
The spatial semidiscretizations and the Petviashvili method
use SBP operators provided by the package
SummationByPartsOperators.jl \cite{ranocha2021sbp}, wrapping
FFTW \cite{frigo2005design} for Fourier methods.
We use CairoMakie.jl \cite{danisch2021makie} to visualize the results.
The code to reproduce all numerical experiments is available in
our reproducibility repository \cite{biswas2024travelingRepro}.

\subsection{Numerical tests of asymptotic preservation} \label{sec:test-AP}
We begin by comparing the solution of the
KdV equation with those of the KdVH system for different values of $\tau$. Specifically, we
consider the KdV and the KdVH system on the domain $[x_{L}, x_{R}] \times (0, T] = [-40, 40] \times (0, 80]$
with periodic boundary conditions. We consider approximation of the soliton
solution of the KdV equation:
\begin{align}
    \ukdv(x,t) = A \sech^2\left(\frac{\sqrt{3A}(x-ct)}{6}\right).
\end{align}
We use well-prepared initial data, meaning that we set
$u(x,0) = \ukdv(x,0)$, $\pmb{v}(0) = D_- \pmb{\ukdv}(0)$ and $\pmb{w}(0) = D D_- \pmb{\ukdv}(0)$,
where $D_\pm$ and $D$ are upwind SBP derivative operators.

We discretize the spatial derivatives in both the KdV and KdVH equations
using periodic FD SBP operators with $N = 2^8$ grid points.  The semi-discrete system is
integrated in time using the type II ARK method ARS(4,4,3) with a time step size of $\Delta t = 0.01$.
With these fixed spatial and temporal parameters, Figure \ref{fig:Kdvh_convg_kdv_ARS443} illustrates
the convergence of the KdVH solution to that of the KdV equation as the relaxation
parameter $\tau$ decreases.

\begin{figure}[ht]
    \centering
    \includegraphics[width=4in]{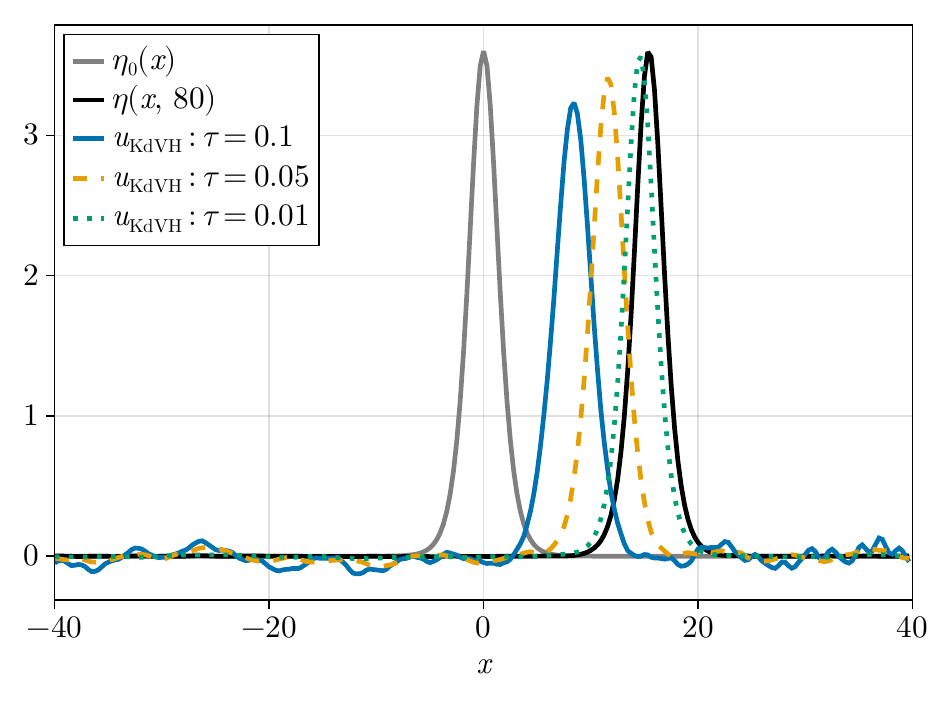}
    \caption{Comparison of the solution of the KdV equation \eqref{kdv} (denoted by $\ukdv(x,80)$) and its hyperbolic
    approximation \eqref{kdvH} (denoted by '$u_\mathrm{KdVH}: \tau$') at time $t=80$. In all cases, the solution is computed using a spatial
    discretization based on periodic SBP operators and the ARK method ARS(4,4,3) in time. The numerical solutions of
    the KdVH system converge to the numerical solution of the KdV equation as $\tau \to 0$.}
    \label{fig:Kdvh_convg_kdv_ARS443}
\end{figure}

To quantitatively validate the AP property of the schemes developed in Section \ref{sec:AP-analysis},
we examine the $\ell_2$ norm induced by the mass/norm matrix $M$ of the SBP
operator used in space of the differences $\pmb{u} - \pmb{\ukdv}$, $\pmb{v} - D_- \pmb{\ukdv}$,
and $\pmb{w} - D D_-\pmb{\ukdv}$ , where $\pmb{u}$, $\pmb{v}$, and $\pmb{w}$
represent the solutions of the fully discretized
KdVH system, and $\pmb{\ukdv}$ is a numerical solution of the KdV equation.
 For the numerical solutions of both the KdV and KdVH equations,
we use spatial discretization based on periodic SBP operators on the domain $[-40,40]$ with $2^{10}$ grid points.
Below, we present the asymptotic errors for various ImEx time integrators, with four methods from each of type I and type II.
In all cases, the time integrators use a fixed step size $\Delta t = 0.005$, and
all errors are computed with respect to the numerical solution of the KdV equation at the final time $t = 16.67$.

\subsubsection{AP results for type I ImEx-RK methods}
We begin with type I methods, denoted by NAME($s_E, s_I, p$), where the triplet
$(s_E, s_I, p)$ specifies the number of stages in the explicit part ($s_E$),
the number of stages in the implicit part ($s_I$), and the overall order ($p$)
of the ImEx-RK method.

First, we consider a second-order type I method, denoted by SSP2-ImEx(2,2,2),
as defined in Table~\ref{tab:SSP2ImEx222}. This method has an implicit part that is
not SA and an explicit part that is not FSAL, meaning it is not GSA. However,
the overall method is $L$-stable \cite{boscarino2024asymptotic}. Table \ref{Table:SSP2-222_errors_num}
presents the asymptotic errors in different
variables. The scheme exhibits linear convergence of $\pmb{u}$
as $\tau \to 0$, confirming the AP property for the $\pmb{u}$-component, while for the variables $\pmb{v}$
and $\pmb{w}$, we do not observe the AP property of the method. This observation is in agreement
with Theorem \ref{Th:AP_result_type_I} and Remark \ref{rem:no-GSA}, since this method does not have the GSA property.

 \begin{table}[htbp]
    \centering
    \caption{Asymptotic errors and estimated orders of convergence (EOC) for the variables $\pmb{u}$, $\pmb{v}$, and $\pmb{w}$ when integrating
    the KdVH system using the SSP2-ImEx(2,2,2) method. The $\ell_2$ norms of the errors are calculated relative to
    the numerical solution of the KdV equation, $\pmb{\ukdv}$.}
    \label{Table:SSP2-222_errors_num}
    \begin{tabular}{ccccccc}
      \hline
      \hline
      \multicolumn{1}{c}{$\tau$} &
      \multicolumn{1}{c}{$||\pmb{u}-\pmb{\ukdv}||_2$} &
      \multicolumn{1}{c}{EOC $\pmb{u}$} &
      \multicolumn{1}{c}{$||\pmb{v}-D_- \pmb{\ukdv}||_2$} &
      \multicolumn{1}{c}{EOC $\pmb{v}$} &
      \multicolumn{1}{c}{$||\pmb{w}-D D_- \pmb{\ukdv}||_2$} &
      \multicolumn{1}{c}{EOC $\pmb{w}$} \\\hline
      1.00e-01 & 3.77e+00 &  & 3.13e+00 &  & 3.61e+00 &  \\
      1.00e-03 & 5.35e-02 & 0.92 & 4.47e-02 & 0.92 & 5.55e-02 & 0.91 \\
      1.00e-05 & 5.37e-04 & 1.00 & 8.96e-03 & 0.35 & 1.63e-02 & 0.27 \\
      1.00e-07 & 5.37e-06 & 1.00 & 9.51e-03 & -0.01 & 1.69e-02 & -0.01 \\
      1.00e-09 & 5.39e-08 & 1.00 & 9.52e-03 & -0.00 & 1.69e-02 & -0.00 \\
    \end{tabular}
\end{table}

We now examine another second-order type I method, SSP2-ImEx(3,3,2), as defined in Table~\ref{tab:SSP2ImEx332}.
This method has an implicit part that is SA and an explicit part that does not satisfy FSAL, which means
it does not satisfy the GSA property. However, the overall method remains $L$-stable due to the SA property of the
implicit part \cite{boscarino2024asymptotic}. Table~\ref{Table:SSP2-332_errors_num} shows the asymptotic
errors for different variables. The convergence for the $\pmb{u}$ variable demonstrates linear convergence
as $\tau \to 0$, indicating the AP property for the $\pmb{u}$-component. Additionally, compared to the
SSP2-ImEx(2,2,2) method, the SSP2-ImEx(3,3,2) method shows improved convergence for the
algebraic variables $\pmb{v}$ and $\pmb{w}$, despite both methods lacking the GSA property. We hypothesize that this improvement arises
because the SSP2-ImEx(3,3,2) method only violates the FSAL property.

\begin{table}[htbp]
    \centering
    \caption{Asymptotic errors and estimated orders of convergence (EOC) for the variables $\pmb{u}$, $\pmb{v}$, and $\pmb{w}$ when integrating
    the KdVH system using the SSP2-ImEx(3,3,2) method. The $\ell_2$ norms of the errors are calculated relative to
    the numerical solution of the KdV equation, $\pmb{\ukdv}$.}
    \label{Table:SSP2-332_errors_num}
    \begin{tabular}{ccccccc}
      \hline
      \hline
      \multicolumn{1}{c}{$\tau$} &
      \multicolumn{1}{c}{$||\pmb{u}-\pmb{\ukdv}||_2$} &
      \multicolumn{1}{c}{EOC $\pmb{u}$} &
      \multicolumn{1}{c}{$||\pmb{v}-D_- \pmb{\ukdv}||_2$} &
      \multicolumn{1}{c}{EOC $\pmb{v}$} &
      \multicolumn{1}{c}{$||\pmb{w}-D D_- \pmb{\ukdv}||_2$} &
      \multicolumn{1}{c}{EOC $\pmb{w}$} \\\hline
      1.00e-01 & 3.77e+00 &  & 3.13e+00 &  & 3.61e+00 &  \\
      1.00e-03 & 5.35e-02 & 0.92 & 4.88e-02 & 0.90 & 6.31e-02 & 0.88 \\
      1.00e-05 & 5.36e-04 & 1.00 & 4.92e-04 & 1.00 & 6.46e-04 & 0.99 \\
      1.00e-07 & 5.36e-06 & 1.00 & 1.23e-05 & 0.80 & 2.61e-05 & 0.70 \\
      1.00e-09 & 5.38e-08 & 1.00 & 1.20e-05 & 0.01 & 2.62e-05 & -0.00 \\
    \end{tabular}
\end{table}

Now we consider a second-order type I method AGSA(3,4,2) given in Table~\ref{tab:AGSA342}
that satisfies the GSA property. Table~\ref{Table:AGSA-342_errors_num} shows that we obtain AP
properties for all the variables $\pmb{u}$, $\pmb{v}$, and $\pmb{w}$, supporting our theoretical results for
type I methods with the GSA property.

\begin{table}[htbp]
    \centering
    \caption{Asymptotic errors and estimated orders of convergence (EOC) for the variables $\pmb{u}$, $\pmb{v}$, and $\pmb{w}$ when integrating
    the KdVH system using the AGSA(3,4,2) method. The $\ell_2$ norms of the errors are calculated relative to the
    numerical solution of the KdV equation, $\pmb{\ukdv}$.}
    \label{Table:AGSA-342_errors_num}
    \begin{tabular}{ccccccc}
      \hline
      \hline
      \multicolumn{1}{c}{$\tau$} &
      \multicolumn{1}{c}{$||\pmb{u}-\pmb{\ukdv}||_2$} &
      \multicolumn{1}{c}{EOC $\pmb{u}$} &
      \multicolumn{1}{c}{$||\pmb{v}-D_- \pmb{\ukdv}||_2$} &
      \multicolumn{1}{c}{EOC $\pmb{v}$} &
      \multicolumn{1}{c}{$||\pmb{w}-D D_- \pmb{\ukdv}||_2$} &
      \multicolumn{1}{c}{EOC $\pmb{w}$} \\\hline
      1.00e-01 & 3.80e+00 &  & 3.16e+00 &  & 3.63e+00 &  \\
      1.00e-03 & 5.93e-02 & 0.90 & 5.44e-02 & 0.88 & 7.12e-02 & 0.85 \\
      1.00e-05 & 5.94e-04 & 1.00 & 5.50e-04 & 1.00 & 7.32e-04 & 0.99 \\
      1.00e-07 & 5.94e-06 & 1.00 & 5.49e-06 & 1.00 & 7.30e-06 & 1.00 \\
      1.00e-09 & 5.96e-08 & 1.00 & 7.93e-08 & 0.92 & 9.60e-08 & 0.94 \\
    \end{tabular}
\end{table}

To test a high-order method of type I, we consider a third-order $L$-stable type I method SSP3-ImEx(3,4,3) given in
Table~\ref{tab:SSP3ImEx343}, which is neither SA in the implicit part nor FSAL in the explicit part.
Since this method does not possess the GSA property, we only expect the AP property in the $\pmb{u}$-component, as shown in
Table~\ref{Table:SSP3-343_errors_num}.

\begin{table}[htbp]
    \centering
    \caption{Asymptotic errors and estimated orders of convergence (EOC) for the variables $\pmb{u}$, $\pmb{v}$, and $\pmb{w}$ when integrating
    the KdVH system using the SSP3-ImEx(3,4,3) method. The $\ell_2$ norms of the errors are calculated relative to
    the numerical solution of the KdV equation, $\pmb{\ukdv}$.}
    \label{Table:SSP3-343_errors_num}
    \begin{tabular}{ccccccc}
      \hline
      \hline
      \multicolumn{1}{c}{$\tau$} &
      \multicolumn{1}{c}{$||\pmb{u}-\pmb{\ukdv}||_2$} &
      \multicolumn{1}{c}{EOC $\pmb{u}$} &
      \multicolumn{1}{c}{$||\pmb{v}-D_- \pmb{\ukdv}||_2$} &
      \multicolumn{1}{c}{EOC $\pmb{v}$} &
      \multicolumn{1}{c}{$||\pmb{w}-D D_- \pmb{\ukdv}||_2$} &
      \multicolumn{1}{c}{EOC $\pmb{w}$} \\\hline
      1.00e-01 & 3.77e+00 &  & 3.13e+00 &  & 3.60e+00 &  \\
      1.00e-03 & 5.34e-02 & 0.92 & 4.88e-02 & 0.90 & 6.25e-02 & 0.88 \\
      1.00e-05 & 5.36e-04 & 1.00 & 3.73e-03 & 0.56 & 7.22e-03 & 0.47 \\
      1.00e-07 & 5.36e-06 & 1.00 & 3.82e-03 & -0.01 & 6.80e-03 & 0.01 \\
      1.00e-09 & 5.38e-08 & 1.00 & 3.83e-03 & -0.00 & 6.81e-03 & -0.00 \\
    \end{tabular}
\end{table}

\subsubsection{AP results for type II ImEx-RK methods}
In this section, we present the quantitative asymptotic errors for different type II methods.
First, we consider a second-order ARS(2,2,2) method (in Table~\ref{tab:ARS222}) and a third-order ARS(4,4,3)
method (in Table~\ref{tab:ARS443}), both of which are type II methods that satisfy the GSA property.
We simulate the semi-discretized system with well-prepared initial data and present the asymptotic errors in
Table~\ref{Table:ARS222_errors_num} and Table~\ref{Table:ARS443_errors_num}, respectively, for these two methods.
The convergence rates in these tables illustrate the AP property for all the components with these two methods,
as expected according to our theoretical results.

\begin{table}[htbp]
\centering
\caption{Asymptotic errors and estimated orders of convergence (EOC) for the variables $\pmb{u}$, $\pmb{v}$,
 and $\pmb{w}$ when integrating the KdVH system using the ARS(2,2,2) method. The $\ell_2$ norms
  of the errors are calculated relative to the numerical solution of the KdV equation, $\pmb{\ukdv}$.}
\label{Table:ARS222_errors_num}
\begin{tabular}{ccccccc}
  \hline
  \hline
  \multicolumn{1}{c}{$\tau$} &
  \multicolumn{1}{c}{$||\pmb{u}-\pmb{\ukdv}||_2$} &
  \multicolumn{1}{c}{EOC $\pmb{u}$} &
  \multicolumn{1}{c}{$||\pmb{v}-D_- \pmb{\ukdv}||_2$} &
  \multicolumn{1}{c}{EOC $\pmb{v}$} &
  \multicolumn{1}{c}{$||\pmb{w}-D D_- \pmb{\ukdv}||_2$} &
  \multicolumn{1}{c}{EOC $\pmb{w}$} \\\hline
  1.00e-01 & 3.77e+00 &  & 3.13e+00 &  & 3.61e+00 &  \\
  1.00e-03 & 5.35e-02 & 0.92 & 4.88e-02 & 0.90 & 6.31e-02 & 0.88 \\
  1.00e-05 & 5.36e-04 & 1.00 & 4.89e-04 & 1.00 & 6.30e-04 & 1.00 \\
  1.00e-07 & 5.36e-06 & 1.00 & 4.89e-06 & 1.00 & 6.30e-06 & 1.00 \\
  1.00e-09 & 5.38e-08 & 1.00 & 5.01e-08 & 0.99 & 6.44e-08 & 1.00 \\
\end{tabular}
\end{table}

\begin{table}[htbp]
\centering
\caption{Asymptotic errors and estimated orders of convergence (EOC) for the variables $\pmb{u}$, $\pmb{v}$,
 and $\pmb{w}$ when integrating the KdVH system using the ARS(4,4,3) method
 in time and an upwind FD method in space. The $\ell_2$ norms
  of the errors are calculated relative to the numerical solution of the KdV equation, $\pmb{\ukdv}$.}
\label{Table:ARS443_errors_num}
\begin{tabular}{ccccccc}
  \hline
  \hline
  \multicolumn{1}{c}{$\tau$} &
  \multicolumn{1}{c}{$||\pmb{u}-\pmb{\ukdv}||_2$} &
  \multicolumn{1}{c}{EOC $\pmb{u}$} &
  \multicolumn{1}{c}{$||\pmb{v}-D_- \pmb{\ukdv}||_2$} &
  \multicolumn{1}{c}{EOC $\pmb{v}$} &
  \multicolumn{1}{c}{$||\pmb{w}-D D_- \pmb{\ukdv}||_2$} &
  \multicolumn{1}{c}{EOC $\pmb{w}$} \\\hline
  1.00e-01 & 3.76e+00 &  & 3.13e+00 &  & 3.60e+00 &  \\
  1.00e-03 & 5.34e-02 & 0.92 & 4.88e-02 & 0.90 & 6.31e-02 & 0.88 \\
  1.00e-05 & 5.36e-04 & 1.00 & 4.89e-04 & 1.00 & 6.31e-04 & 1.00 \\
  1.00e-07 & 5.36e-06 & 1.00 & 4.89e-06 & 1.00 & 6.31e-06 & 1.00 \\
  1.00e-09 & 5.38e-08 & 1.00 & 5.23e-08 & 0.99 & 6.50e-08 & 0.99 \\
\end{tabular}
\end{table}

To demonstrate that the analysis of the properties is not limited to
finite difference methods in space, we have also considered a DG method
in space. The results for the ARS(4,4,3) method with a DG method using
polynomials of degree $3$ with $2^8$ elements are presented in
Table~\ref{Table:ARS443_DG_errors_num}.
We obtain similar results for Fourier pseudospectral methods,
which are not shown here.

\begin{table}[htbp]
\centering
\caption{Asymptotic errors and estimated orders of convergence (EOC) for the variables $\pmb{u}$, $\pmb{v}$,
 and $\pmb{w}$ when integrating the KdVH system using the ARS(4,4,3) method
 in time and a DG method in space. The $\ell_2$ norms
  of the errors are calculated relative to the numerical solution of the KdV equation, $\pmb{\ukdv}$.}
\label{Table:ARS443_DG_errors_num}
\begin{tabular}{ccccccc}
    \hline
    \hline
    \multicolumn{1}{c}{$\tau$} &
    \multicolumn{1}{c}{$||\pmb{u}-\pmb{\ukdv}||_2$} &
    \multicolumn{1}{c}{EOC $\pmb{u}$} &
    \multicolumn{1}{c}{$||\pmb{v}-D_- \pmb{\ukdv}||_2$} &
    \multicolumn{1}{c}{EOC $\pmb{v}$} &
    \multicolumn{1}{c}{$||\pmb{w}-D D_- \pmb{\ukdv}||_2$} &
    \multicolumn{1}{c}{EOC $\pmb{w}$} \\\hline
    1.00e-01 & 3.76e+00 &  & 3.13e+00 &  & 3.60e+00 &  \\
    1.00e-03 & 5.34e-02 & 0.92 & 4.88e-02 & 0.90 & 6.31e-02 & 0.88 \\
    1.00e-05 & 5.36e-04 & 1.00 & 4.89e-04 & 1.00 & 6.31e-04 & 1.00 \\
    1.00e-07 & 5.36e-06 & 1.00 & 4.89e-06 & 1.00 & 6.31e-06 & 1.00 \\
    1.00e-09 & 5.36e-08 & 1.00 & 6.20e-08 & 0.95 & 6.39e-08 & 1.00 \\
  \end{tabular}
  \end{table}

Tables \ref{Table:ARK3(2)4L[2]SA_errors_num} and \ref{Table:ARK4(3)6L[2]SA_errors_num} present the asymptotic
errors for two methods: the third-order ARK3(2)4L[2]SA (in Table~\ref{tab:ARK324L2SA}) and the fourth-order
ARK4(3)6L[2]SA (in Table~\ref{tab:ARK436L2SA}) of type II,
as proposed by Kennedy and Carpenter \cite{kennedy2003additive}. Both methods do not have the GSA property,
but their implicit parts are SA.
The simulations are performed with well-prepared initial data, and the convergence rates shown in the tables
indicate that we consistently obtain the AP property for the $\pmb{u}$-component. It appears that we observe linear
convergence rates for the algebraic variables within a certain range of $\tau$ values, as shown in the table,
particularly for higher-order ImEx methods. This behavior surpasses the guarantees provided by the
theoretical results for such methods. The observed effect is attributed to the use of a sufficiently
small time step in the simulation. However, with a larger time step,
we observe that the AP property manifests only
in the $\pmb{u}$-component, not in the algebraic variables.

\begin{table}[htbp]
\centering
\caption{Asymptotic errors and estimated orders of convergence (EOC) for the variables $\pmb{u}$, $\pmb{v}$, and $\pmb{w}$ when integrating
the KdVH system using the ARK3(2)4L[2]SA method. The $\ell_2$ norms of the errors are calculated relative to the
numerical solution of the KdV equation, $\pmb{\ukdv}$.}
\label{Table:ARK3(2)4L[2]SA_errors_num}
\begin{tabular}{ccccccc}
  \hline
  \hline
  \multicolumn{1}{c}{$\tau$} &
  \multicolumn{1}{c}{$||\pmb{u}-\pmb{\ukdv}||_2$} &
  \multicolumn{1}{c}{EOC $\pmb{u}$} &
  \multicolumn{1}{c}{$||\pmb{v}-D_- \pmb{\ukdv}||_2$} &
  \multicolumn{1}{c}{EOC $\pmb{v}$} &
  \multicolumn{1}{c}{$||\pmb{w}-D D_- \pmb{\ukdv}||_2$} &
  \multicolumn{1}{c}{EOC $\pmb{w}$} \\\hline
  1.00e-01 & 3.77e+00 &  & 3.13e+00 &  & 3.60e+00 &  \\
  1.00e-03 & 5.34e-02 & 0.92 & 4.88e-02 & 0.90 & 6.32e-02 & 0.88 \\
  1.00e-05 & 5.36e-04 & 1.00 & 4.90e-04 & 1.00 & 6.37e-04 & 1.00 \\
  1.00e-07 & 5.36e-06 & 1.00 & 1.27e-05 & 0.79 & 2.69e-05 & 0.69 \\
  1.00e-09 & 5.38e-08 & 1.00 & 1.24e-05 & 0.01 & 2.71e-05 & -0.00 \\
\end{tabular}
\end{table}

\begin{table}[htbp]
\centering
\caption{Asymptotic errors and estimated orders of convergence (EOC) for the variables $\pmb{u}$, $\pmb{v}$, and $\pmb{w}$ when integrating
the KdVH system using the ARK4(3)6L[2]SA method. The $\ell_2$ norms of the errors are calculated relative to the
numerical solution of the KdV equation, $\pmb{\ukdv}$.}
\label{Table:ARK4(3)6L[2]SA_errors_num}
\begin{tabular}{ccccccc}
  \hline
  \hline
  \multicolumn{1}{c}{$\tau$} &
  \multicolumn{1}{c}{$||\pmb{u}-\pmb{\ukdv}||_2$} &
  \multicolumn{1}{c}{EOC $\pmb{u}$} &
  \multicolumn{1}{c}{$||\pmb{v}-D_- \pmb{\ukdv}||_2$} &
  \multicolumn{1}{c}{EOC $\pmb{v}$} &
  \multicolumn{1}{c}{$||\pmb{w}-D D_- \pmb{\ukdv}||_2$} &
  \multicolumn{1}{c}{EOC $\pmb{w}$} \\\hline
  1.00e-01 & 3.77e+00 &  & 3.13e+00 &  & 3.60e+00 &  \\
  1.00e-03 & 5.34e-02 & 0.92 & 4.88e-02 & 0.90 & 6.31e-02 & 0.88 \\
  1.00e-05 & 5.36e-04 & 1.00 & 4.89e-04 & 1.00 & 6.31e-04 & 1.00 \\
  1.00e-07 & 5.36e-06 & 1.00 & 4.89e-06 & 1.00 & 6.31e-06 & 1.00 \\
  1.00e-09 & 5.38e-08 & 1.00 & 4.97e-08 & 1.00 & 6.35e-08 & 1.00 \\
\end{tabular}
\end{table}

\subsubsection{Asymptotic-accuracy property}
The various classes of ImEx methods employed in the numerical experiments presented here have been proven to satisfy the AP
and AA properties for hyperbolic relaxation systems, as established in \cite{pareschi2005implicit}. Analogous results hold
for the relaxation system considered in this study, as demonstrated in Section~\ref{sec:AP-analysis}. In this section, we numerically investigate
the AA property for different classes of ImEx methods with two values of the relaxation parameter, capturing distinct
regimes of the relaxation limit. Specifically, we select $\tau \in \{10^{-5}, 10^{-9}\}$ and demonstrate error convergence
for each $\tau$ using five methods: AGSA(3,4,2), SSP3-ImEx(3,4,3), ARS(2,2,2), ARS(4,4,3), and ARK3(2)4L[2]SA.

For all experiments, we use an 8th-order periodic first-derivative upwind SBP operator with $N = 2^{10}$ spatial grid
points on the domain $[-40, 40]$ for spatial semi-discretization. The error convergence for the components $\pmb{u}$, $\pmb{v}$,
and $\pmb{w}$ at $t = 4.8$ is illustrated in Figure~\ref{fig:AA_property}. The error at the final time is computed relative
to reference solutions of the KdVH system denoted by $\pmb{u_P}$, $\pmb{v_P}$, and $\pmb{w_P}$, where $\pmb{u_P}$ is obtained using a Petviashvili-type
method on a fine spatial grid with $2^{11}$ grid points. The reference solution for the auxiliary variable
$\pmb{w_P} = c \pmb{u_P} - \frac{\pmb{u_P}^2}{2}$ is derived by integrating \eqref{Eq:kdvh_wave_ode_a}, and $\pmb{v_P} = D_{-}\pmb{u_P} - c \tau D_{-} \pmb{w_P}$ is
obtained by using \eqref{Eq:kdvh_wave_ode_c}.

In accordance with our theoretical results, the methods AGSA(3,4,2), ARS(2,2,2) and ARS(4,4,3) confirm the AA property
for all components. The method SSP3-ImEx(3,4,3) exhibits the AA property for the $\pmb{u}$-component only but not for the auxiliary
components, aligning with theoretical predictions. Additionally, we observe the AA property for the $\pmb{u}$-component with
the ARK3(2)4L[2]SA method, despite this property not being theoretically guaranteed for this method, a behavior similar
to its AP property.

\begin{figure}[htbp]
    \centering
    \begin{subfigure}[b]{\textwidth}
        \centering
        \includegraphics[height=0.30\textheight]{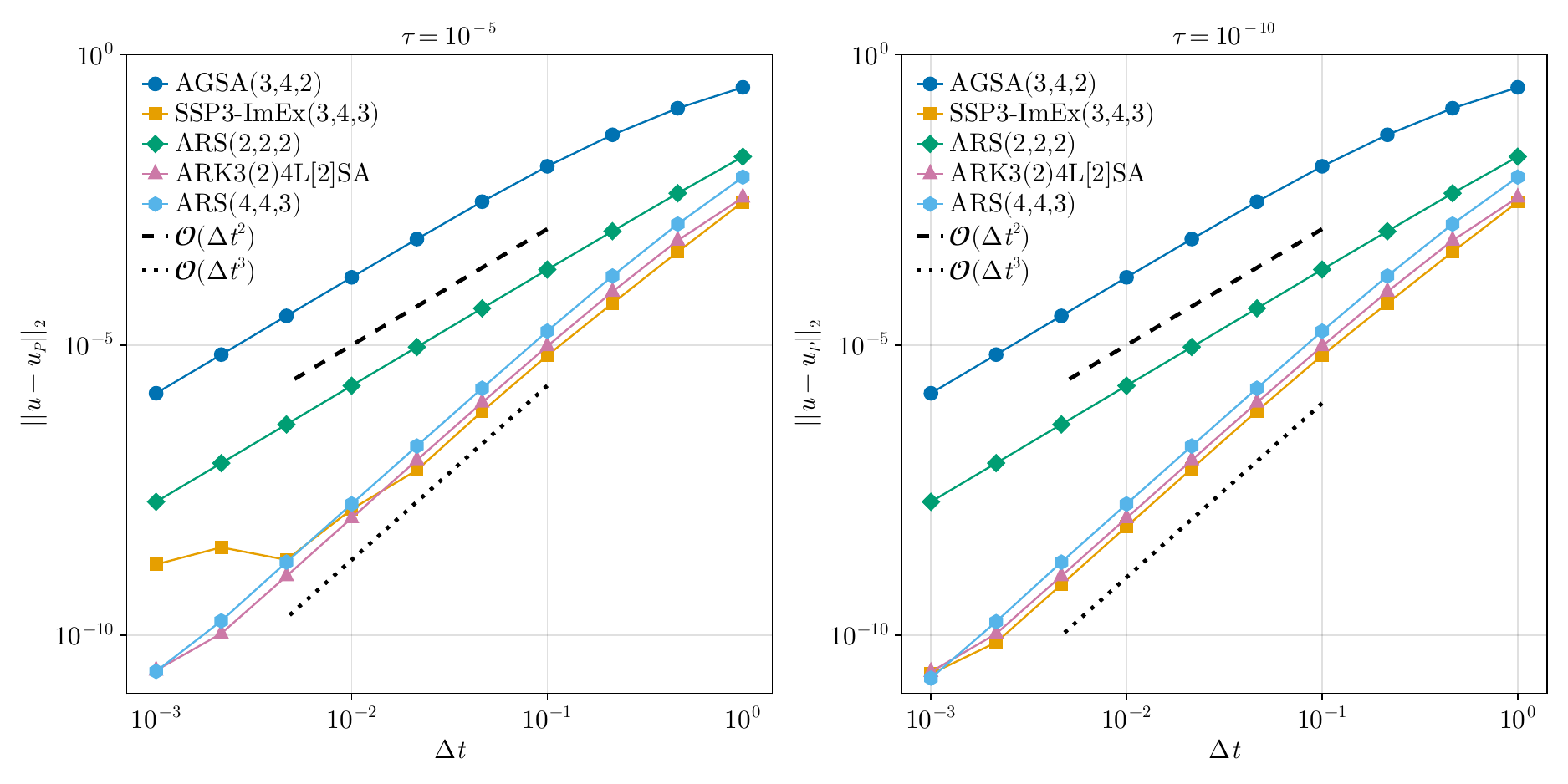}
    \end{subfigure}

    \begin{subfigure}[b]{\textwidth}
        \centering
        \includegraphics[height=0.30\textheight]{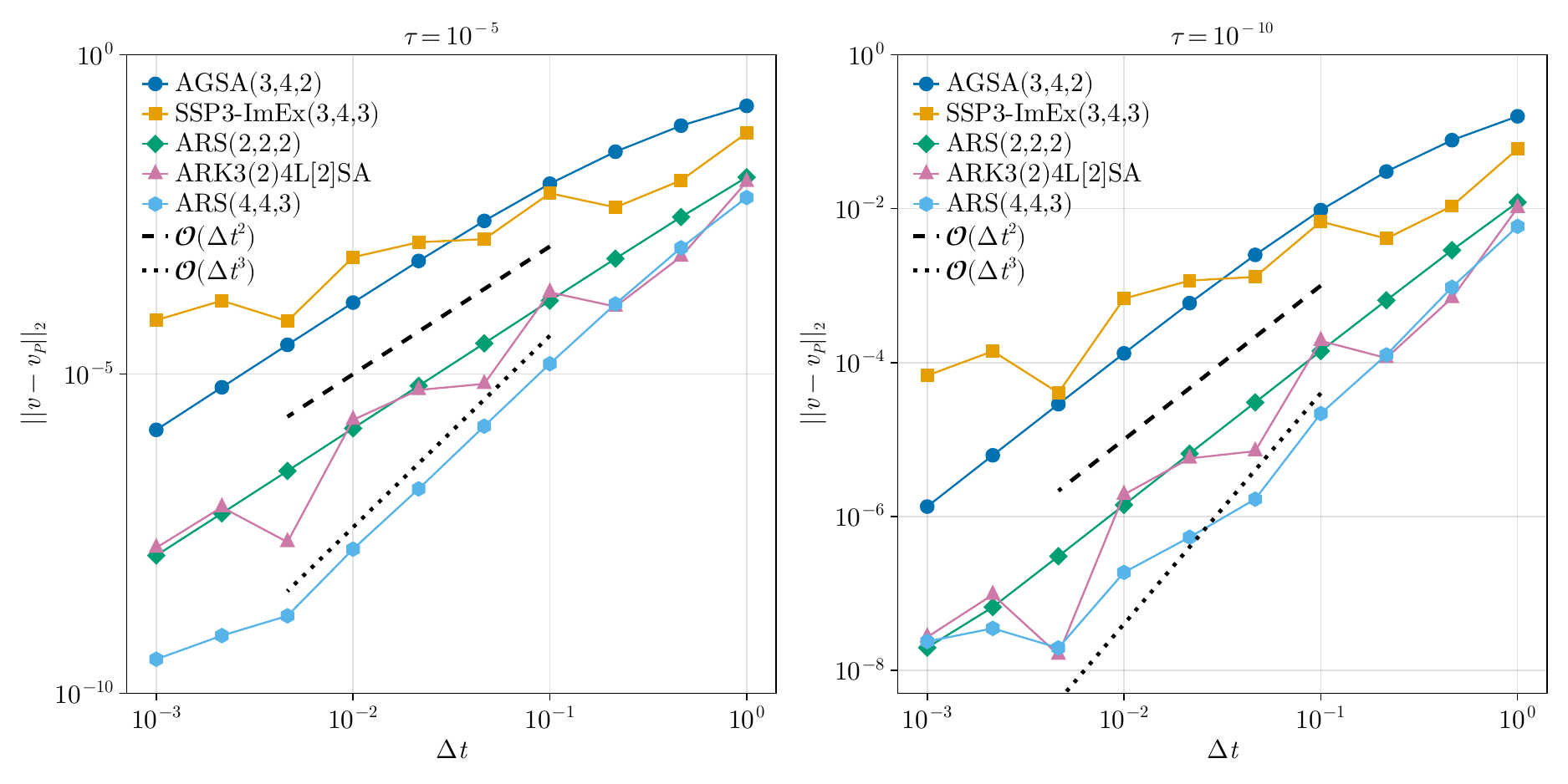}
    \end{subfigure}

    \begin{subfigure}[b]{\textwidth}
        \centering
        \includegraphics[height=0.30\textheight]{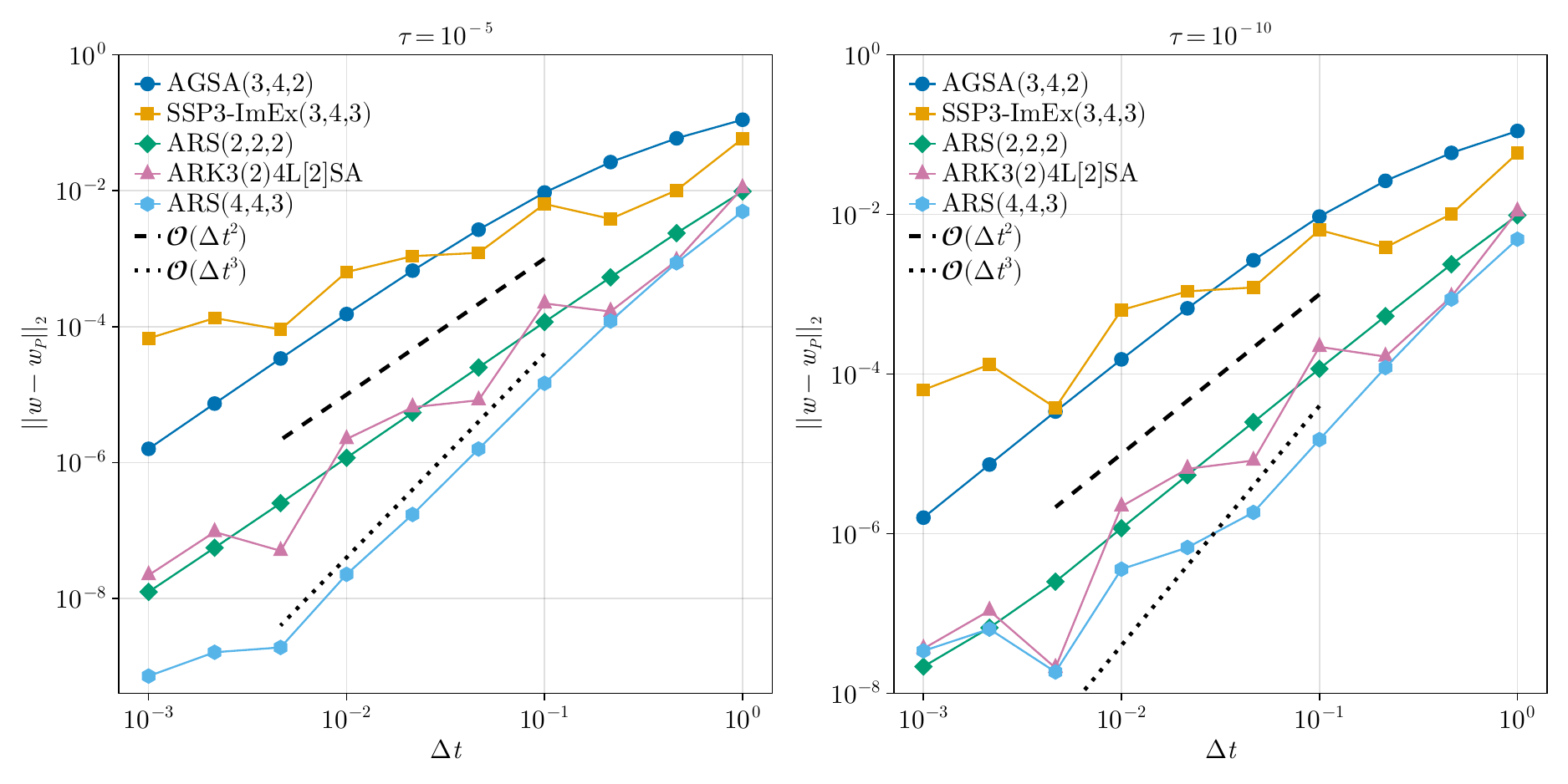}
    \end{subfigure}
    \caption{Error convergence for variables $\pmb{u}$ (top row), $\pmb{v}$ (middle row), and $\pmb{w}$ (bottom row) for two relaxation parameters.
    The reference solutions $\pmb{u_P}$, $\pmb{v_P}$, and $\pmb{w_P}$ are obtained using a Petviashvili-type method and a periodic
    first-derivative SBP operator. The methods AGSA(3,4,2), SSP3-ImEx(3,4,3), ARS(2,2,2), and ARS(4,4,3) exhibit the expected AA property
    for all components, whereas ARK3(2)4L[2]SA demonstrates the AA property for the $\pmb{u}$-component, which is beyond
    our guaranteed theoretical results.}
    \label{fig:AA_property}
\end{figure}

\subsection{Numerical tests of energy conservation} \label{sec:test-conservation}

It is well-known that integrating the KdV equation with an energy-conserving numerical scheme results in linear error
growth over time, whereas a non-conservative method leads to quadratic error growth
\cite{frutos1997accuracy}. This distinction makes conservative
methods superior for maintaining solution accuracy over long time intervals. The KdVH system has a modified
energy that is conserved, and as the relaxation parameter $\tau \to 0$, this energy converges to that of the KdV equation.
We aim to examine the effect of numerically conserving this modified energy on error propagation in the KdVH
system across different values of $\tau$. To achieve this, we use an energy-conserving spatial semi-discretization
combined with a relaxation Runge-Kutta approach, specifically designed to preserve one or more invariants of the system.
For error computation in the KdVH system, we use the analytical solution (or a highly accurate numerical solution)
of the KdV equation as the reference.

\begin{figure}[htbp]
    \centering
    \begin{subfigure}{0.45\textwidth}
        \centering
        \includegraphics[width=\textwidth]{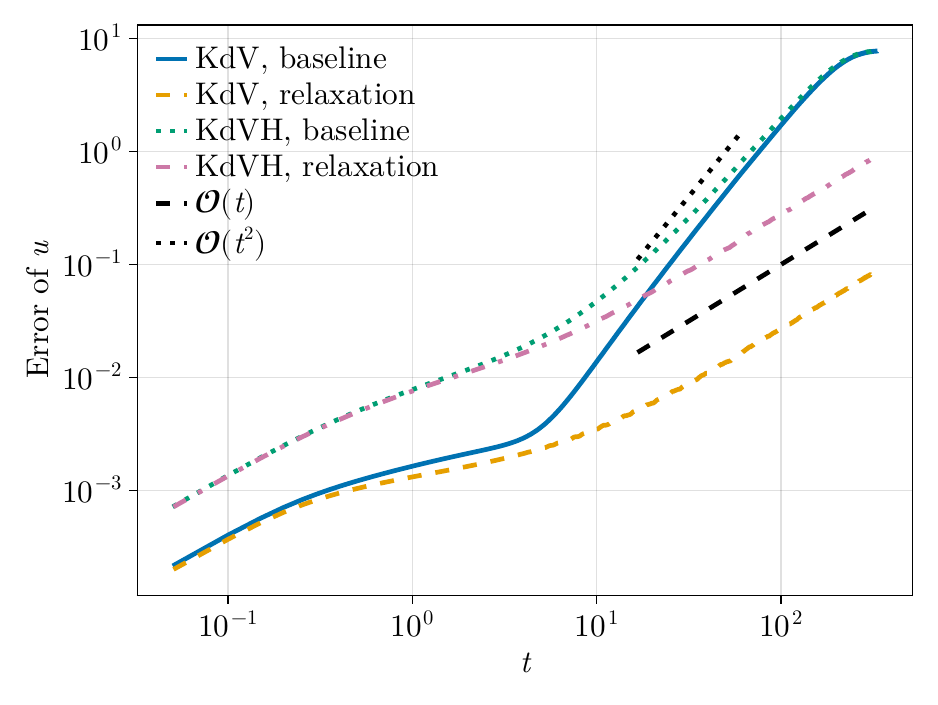}
        \caption{$\tau = 10^{-3}$}
    \end{subfigure}
    \begin{subfigure}{0.45\textwidth}
        \centering
        \includegraphics[width=\textwidth]{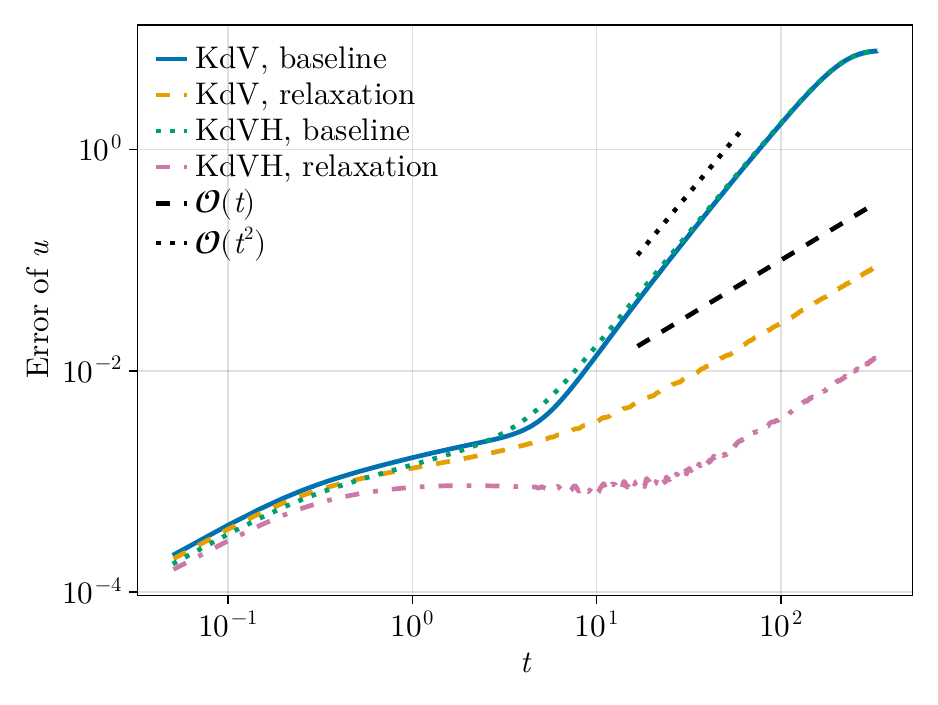}
        \caption{$\tau = 10^{-4}$}
    \end{subfigure}

    \begin{subfigure}{0.45\textwidth}
        \centering
        \includegraphics[width=\textwidth]{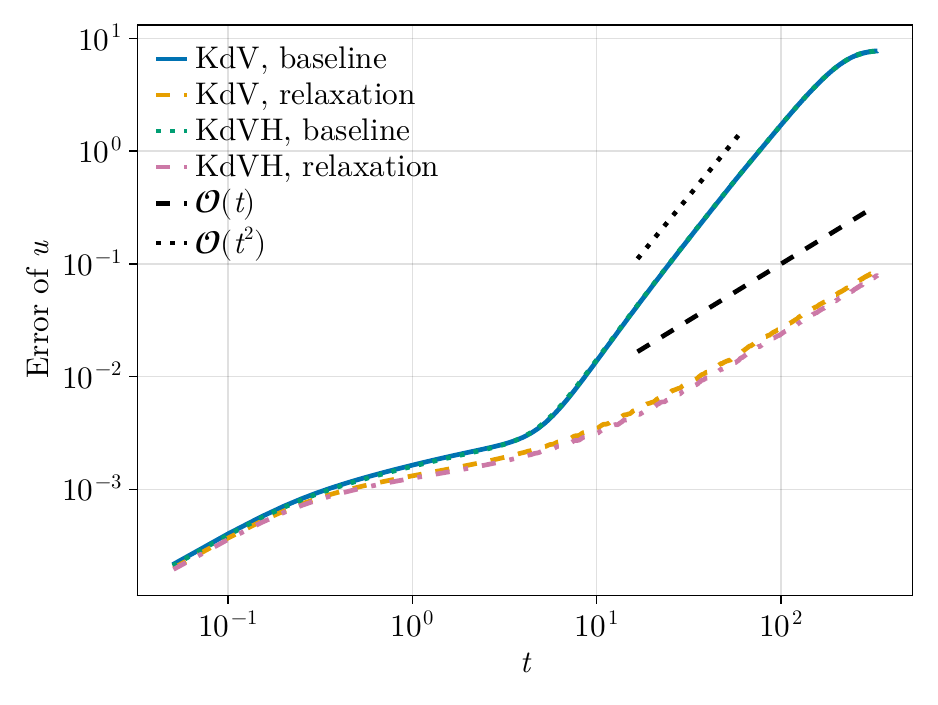}
        \caption{$\tau = 10^{-5}$}
    \end{subfigure}
    \begin{subfigure}{0.45\textwidth}
        \centering
        \includegraphics[width=\textwidth]{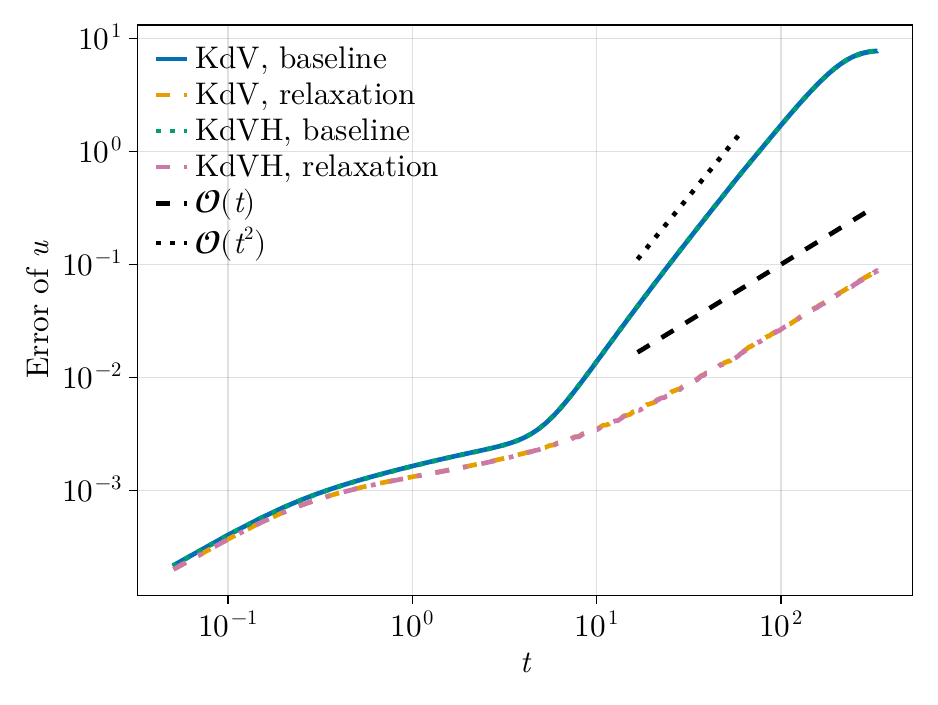}
        \caption{$\tau = 10^{-6}$}
    \end{subfigure}

    \caption{Error growth profiles for the KdVH system up to time $t=333.34$ for four values of the
     relaxation parameter $\tau$. Each subplot compares the numerical solutions with the analytical solitary wave solution of the
     KdV equation. For smaller $\tau$, linear error growth characteristic of conservative methods is observed.
     As $\tau$ decreases, the error growth curves for the KdVH system converge towards those of the KdV system.}
    \label{fig:Error_growth_KdVH_KdV}
\end{figure}

Considering the spatial domain $[-40, 40]$ with $2^8$ grid points and an 8th-order finite-difference operator for derivatives,
we integrate the energy-conserving semi-discretized KdVH system up to a final time of $333.34$ using the ARK method
ARS(4,4,3), with and without entropy relaxation. In each case we start the time stepping with a fixed time step $\Delta t=0.05$. Errors
at each time step are computed with respect to the analytical solution of the KdV equation, and
Figure \ref{fig:Error_growth_KdVH_KdV} presents the error growth profiles for four values of $\tau$.
Each panel in the figure includes reference error growth curves for the KdV equation, demonstrating
linear versus quadratic error growth for conservative versus non-conservative methods. For smaller $\tau$
values, we observe similar behavior in error growth for the KdVH system, while for larger $\tau$ values,
this behavior becomes less evident. Expected linear and quadratic error growth behaviors for conservative
and non-conservative methods, respectively, are observed only when $\tau$ is sufficiently small. Additionally,
as $\tau$ decreases, the error growth curves for the KdVH system converge toward those of the KdV system.

To examine the effects of energy conservation using different ImEx integrators for the KdVH system with different values
of the relaxation parameter $\tau$ with a more challenging solution, we consider the KdV equation and the hyperbolized system with
a 2-soliton solution given by
\begin{align}\label{Eq:2-solition_sol}
    u(x, t) = -\frac{12(\beta_1-\beta_2)\left(\beta_2 \csch^2(\xi_2)+\beta_1 \sech^2(\xi_1)\right)}
    {\left(\sqrt{2 \beta_1} \tanh(\xi_1)-\sqrt{2 \beta_2} \coth(\xi_2)\right)^2} \;.
\end{align}
where $\beta_1 = 0.5$, $\beta_2 = 1$, $\xi_1 = \frac{\sqrt{\beta_1}(x-2 \beta_1 t)}{\sqrt{2}}$, and $\xi_2 = \frac{\sqrt{\beta_2}(x-2 \beta_2 t)}{\sqrt{2}}$.
In this case, we consider the domain $[-60, 100]$ with $2^{10}$ grid points and employ an 8th-order energy-conserving finite-difference
operator for spatial semidiscretization. We integrate the energy-conserving semidiscretized systems for both the KdV
equation and the KdVH system with $\tau = 10^{-3}$ and $\tau = 10^{-5}$ from $t = -20$ to $t = 50$ using different
time-stepping methods, both with and without entropy relaxation. We initialize the simulation from a negative
time to capture the soliton interaction occurring at $t = 0$. Consequently, when plotting error growth over time on a
log-log scale, the time axis is shifted by the starting time. For time integration, we use AGSA(3,4,2), ARS(4,4,3),
and ARK4(3)6L[2]SA, with initial time steps $\Delta t = 0.02, 0.1$, and $0.5$, respectively.

Figure \ref{fig:2_soliton_Error_growth_KdVH_KdV} presents the error growth profiles for the three methods with two
values of $\tau$. At each time step, the error is computed relative to the analytical 2-soliton solution given
by \eqref{Eq:2-solition_sol}. For each method, the time-stepping approach with entropy relaxation exhibits improved
error growth compared to its corresponding baseline method. Notably, all methods display a dip in the error growth
profile during soliton interaction, consistent with previous observations \cite{biswas2023multiple,frutos1997accuracy}.
Furthermore, as $\tau$ decreases, the error growth curves progressively converge toward those of the KdV system.

\begin{figure}[htbp]
    \centering
    \begin{subfigure}{0.45\textwidth}
        \centering
        \includegraphics[width=\textwidth]{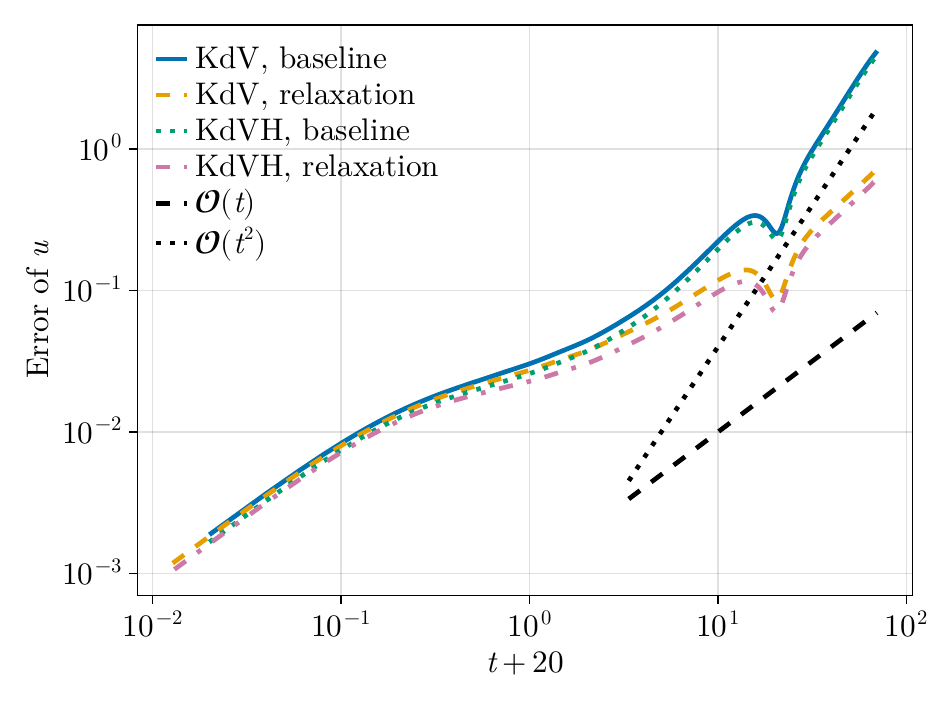}
        \caption{AGSA(3,4,2) and $\tau = 10^{-3}$}
    \end{subfigure}
    \begin{subfigure}{0.45\textwidth}
        \centering
        \includegraphics[width=\textwidth]{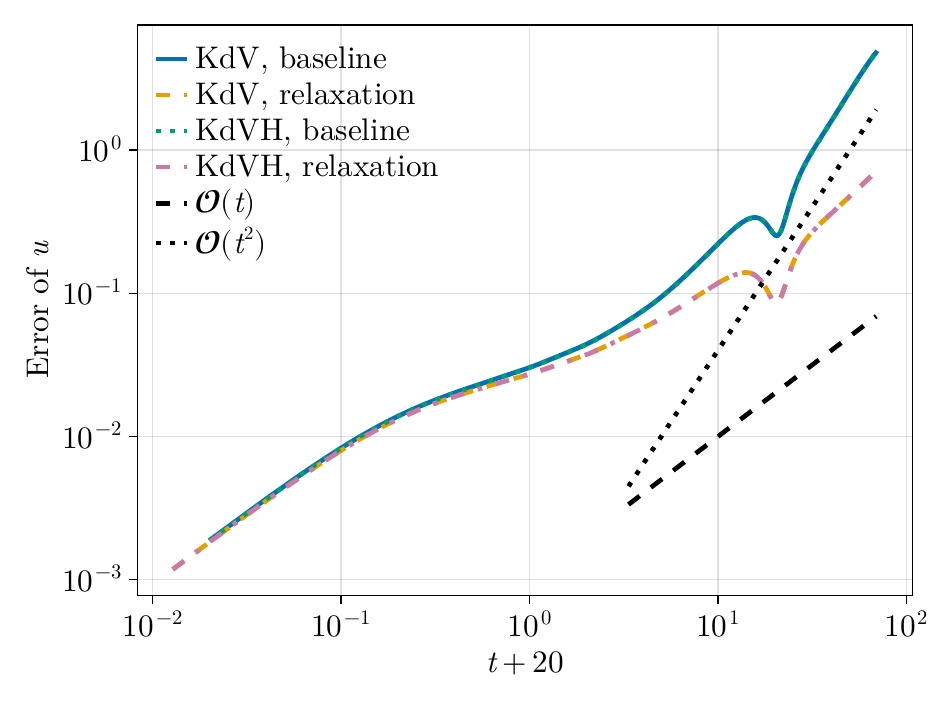}
        \caption{AGSA(3,4,2) and $\tau = 10^{-5}$}
    \end{subfigure}

    \begin{subfigure}{0.45\textwidth}
        \centering
        \includegraphics[width=\textwidth]{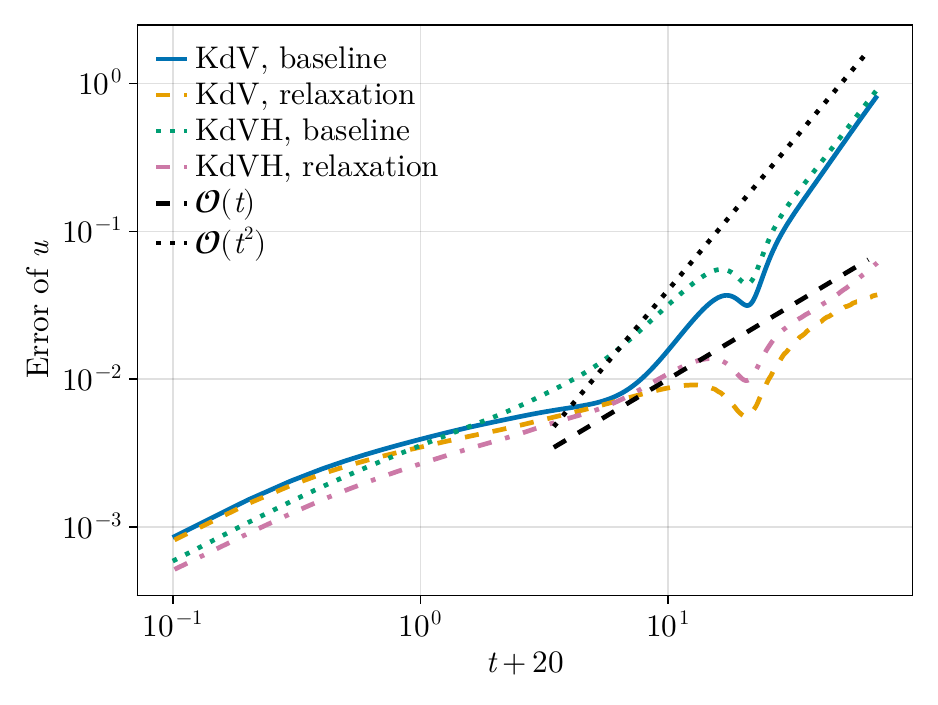}
        \caption{ARS(4,4,3) and $\tau = 10^{-3}$}
    \end{subfigure}
    \begin{subfigure}{0.45\textwidth}
        \centering
        \includegraphics[width=\textwidth]{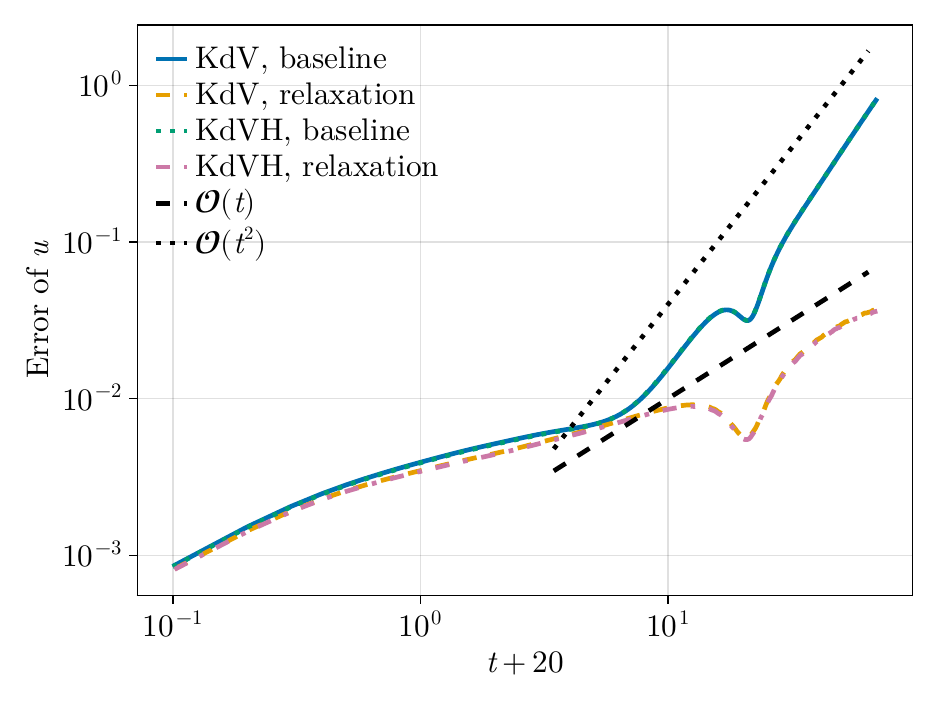}
        \caption{ARS(4,4,3) and $\tau = 10^{-5}$}
    \end{subfigure}

    \begin{subfigure}{0.45\textwidth}
        \centering
        \includegraphics[width=\textwidth]{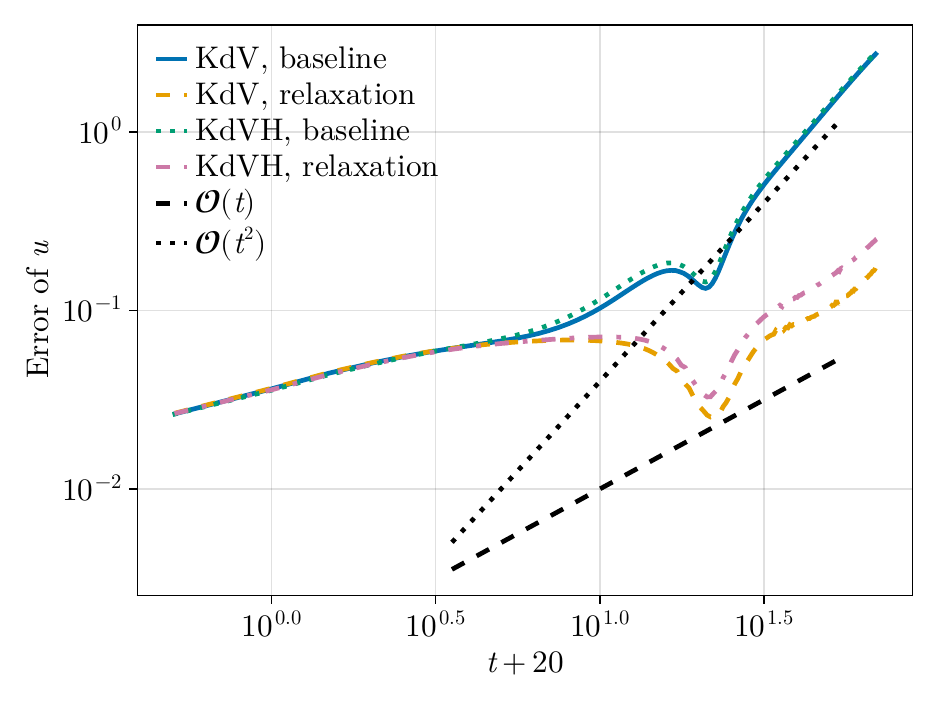}
        \caption{ARK4(3)6L[2]SA and $\tau = 10^{-3}$}
    \end{subfigure}
    \begin{subfigure}{0.45\textwidth}
        \centering
        \includegraphics[width=\textwidth]{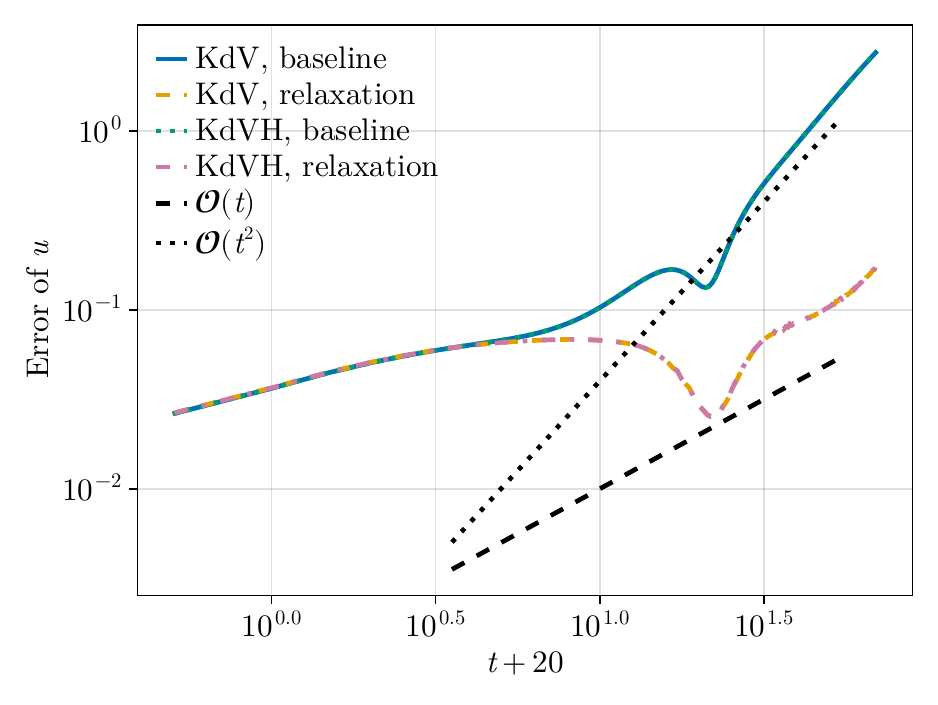}
        \caption{ARK4(3)6L[2]SA and $\tau = 10^{-5}$}
    \end{subfigure}

    \caption{Error growth profiles for the KdVH system with two values of the relaxation parameter $\tau$,
    computed from $t = -20$ to $t = 50$ using three different ImEx methods. Each subplot compares the numerical
    solutions with the analytical 2-soliton solution of the KdV equation. ImEx methods with entropy relaxation exhibit improved
    error growth compared to their corresponding baseline methods. Furthermore, as $\tau$ decreases, the error growth
    curves for the KdVH system converge toward those of the KdV system.}
    \label{fig:2_soliton_Error_growth_KdVH_KdV}
\end{figure}

So far, we have examined the error growth over time by measuring the solution error of the KdVH system with
respect to the analytical solution of the KdV equation. Now, given a particular value of $\tau$, we focus on the
KdVH system itself and compare its numerical solution with the numerically obtained exact solution of the
KdVH system, computed using the Petviashvili method. The Petviashvili method is applied over the domain
$[-40,40]$ with $2^{10}$ grid points to obtain numerically exact solitary wave solutions for the KdVH
system with $\tau = 10^{-2}$. For the numerical solution, we semi-discretize the KdVH using an $8$th-order
upwind finite difference approximation with $2^8$ grid points, resulting in a modified energy-preserving semidiscretization.
Figure~\ref{fig:Error_growth_KdVH(tau)} shows the error growth for two different ImEx-RK methods: SSP2-IMEX(2,2,2),
a type I method, and ARS(2,2,2), a type II method. All time integrations are initialized with a time step of
$\Delta t = 0.05$. Both methods exhibit the expected linear and quadratic error growth over time when integrated
with and without entropy relaxation.

\begin{figure}[htbp]
    \centering
    \begin{subfigure}{0.45\textwidth}
        \centering
        \includegraphics[width=\textwidth]{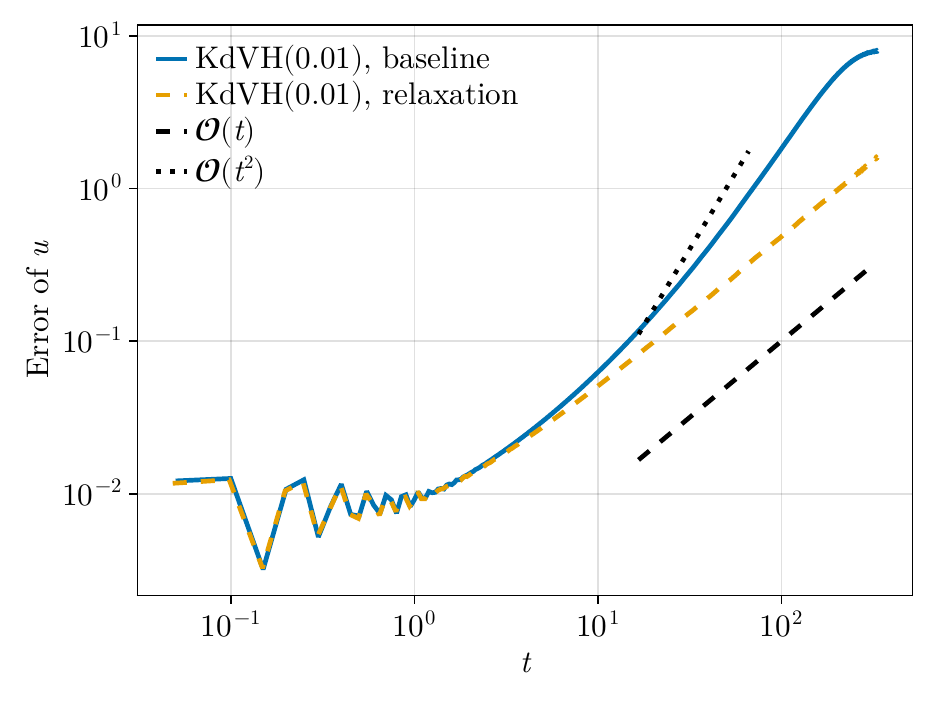}
        \caption{Error growth by SSP2-IMEX(2,2,2)}
    \end{subfigure}
    \begin{subfigure}{0.45\textwidth}
        \centering
        \includegraphics[width=\textwidth]{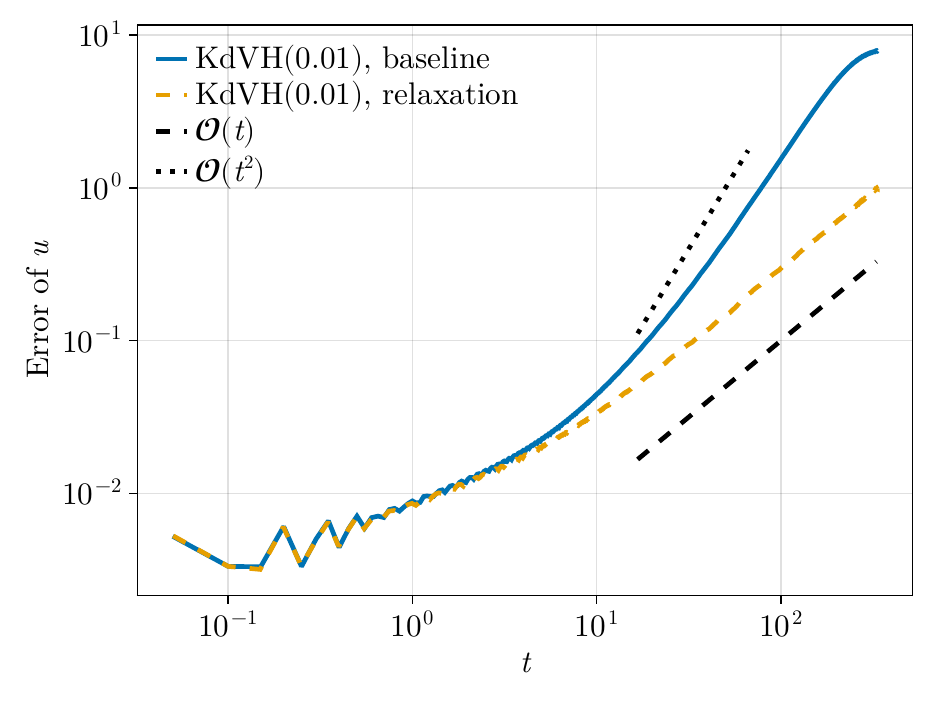}
        \caption{Error growth by ARS(2,2,2)}
    \end{subfigure}

    \caption{Error growth profiles for the KdVH system with $\tau = 10^{-2}$ up to time $t=333.34$. The
     left panel shows the error growth by SSP2-IMEX(2,2,2) with and without entropy relaxation, while the right panel shows
     the error growth by ARS(2,2,2). In each case, the numerical solution is compared with the numerically obtained
     exact solution of the KdVH system.}
    \label{fig:Error_growth_KdVH(tau)}
\end{figure}

\section{Conclusions}

Given the increasing interest in hyperbolic approximations to dispersive nonlinear
wave equations, it is of great interest to understand the dynamics of these
hyperbolic models and develop structure-preserving numerical discretizations
for them.  Here we have carried out this work in relation to the hyperbolized KdV
system.

One of our principal findings is that the dynamics of the KdVH system, studied here
primarily in terms of traveling waves, is in a sense richer than that of the original
KdV equation, and includes additional classes of solitary and periodic waves
including some with lower regularity.  A more extensive investigation of these solutions,
along the lines of \cite{lenells2005traveling}, would be very interesting.
Furthermore, the resemblance of \eqref{DP-like} and its solutions to higher-order water wave models
suggests that there may be a deeper connection between KdVH and such models.

The asymptotic-preserving discretizations developed herein provide essential
guarantees for numerical solutions of KdVH, since in practice one uses a finite
value of the relaxation parameter $\tau$.
Numerical results of asymptotic preservation and asymptotic accuracy presented
herein support our theoretical results, with some ImEx methods by Kennedy and Carpenter
of type II \cite{kennedy2003additive} producing results that outperform theoretical predictions,
suggesting the need for further investigation to fully understand this behavior.
Energy preservation ensures that solutions of KdVH remain closer to those of KdV for longer times.
It would be of interest to investigate the existence of higher-order modified invariants
of KdVH and their numerical preservation.

\section*{Acknowledgments}

AB and DK were supported by the
King Abdullah University of Science and Technology (KAUST).
HR was supported by the Deutsche Forschungsgemeinschaft
(DFG, German Research Foundation, project number 513301895)
and the Daimler und Benz Stiftung (Daimler and Benz foundation,
project number 32-10/22).

\bibliographystyle{plain}
\bibliography{refs}

\appendix

\section{Type I ImEx methods}
\label{app:Tyep_I_ImEx_Methods}
\subsection{SSP2-ImEx(2,2,2)}
\begin{table}[H]
    \centering
    \begin{tabular}{c|cc}
        $0$ & $0$ & $0$ \\
        $1$ & $1$ & $0$ \\
        \hline
        & $1/2$ & $1/2$ \\
    \end{tabular}
    \quad
    \begin{tabular}{c|cc}
        $\gamma$  & $\gamma$  & $0$        \\
        $1 - \gamma$ & $1 - 2\gamma$ & $\gamma$ \\
        \hline
        & $1/2$ & $1/2$ \\
    \end{tabular} \;, with $\gamma = 1-\frac{1}{\sqrt{2}}$.
    \caption{Tableau for the 2nd-order L-stable type I ImEx-RK method: the explicit part is
    not FSAL, the implicit part is not SA, hence not GSA.}
    \label{tab:SSP2ImEx222}
\end{table}

\subsection{SSP2-ImEx(3,3,2)}
\begin{table}[H]
    \centering
    \begin{tabular}{c|ccc}
        $0$ & $0$ & $0$ & $0$ \\
        $1/2$ & $1/2$ & $0$ & $0$ \\
        $1$ & $1/2$ & $1/2$ & $0$ \\
        \hline
        & $1/3$ & $1/3$ & $1/3$ \\
    \end{tabular}
    \quad
    \begin{tabular}{c|ccc}
        $1/4$ & $1/4$ & $0$ & $0$ \\
        $1/4$ & $0$ & $1/4$ & $0$ \\
        $1$ & $1/3$ & $1/3$ & $1/3$ \\
        \hline
        & $1/3$ & $1/3$ & $1/3$ \\
    \end{tabular} \;.
    \caption{Tableau for the 2nd-order L-stable type I ImEx-RK method: the explicit part is not FSAL,
    the implicit part is SA, hence not GSA.}
    \label{tab:SSP2ImEx332}
\end{table}

\subsection{AGSA(3,4,2)}
\begin{table}[H]
\centering
\begin{tabular}{c|cccc}
    $0$ & $0$ & $0$ & $0$ & $0$ \\
    $\tilde{c}_2$ & $\tilde{c}_2$ & $0$ & $0$ & $0$ \\
    $\tilde{c}_3$ & $\tilde{a}_{31}$ & $\tilde{a}_{32}$ & $0$ & $0$ \\
    $1$ & $\tilde{b}_1$ & $\tilde{b}_2$ & $\tilde{b}_3$ & $0$ \\
    \hline
    & $\tilde{b}_1$ & $\tilde{b}_2$ & $\tilde{b}_3$ & $0$ \\
\end{tabular}
\quad
\begin{tabular}{c|cccc}
    $c_1$ & $c_1$ & $0$ & $0$ & $0$ \\
    $c_2$ & $a_{21}$ & $a_{22}$ & $0$ & $0$ \\
    $c_3$ & $a_{31}$ & $a_{32}$ & $a_{33}$ & $0$ \\
    $1$ & $b_1$ & $b_2$ & $b_3$ & $\gamma$ \\
    \hline
     & $b_1$ & $b_2$ & $b_3$ & $\gamma$ \\
\end{tabular} \;.
\caption{Tableau for the 2nd-order type I ImEx-RK method: the explicit part is FSAL, the implicit part is SA, hence GSA.}
\label{tab:AGSA342}
\end{table}

The coefficients in the table are
$c_2 = \tilde{a}_{21} = \frac{-139833537}{38613965}$,
$c_1 = \frac{168999711}{74248304}$,
$\tilde{a}_{31} = \frac{85870407}{49798258}$,
$\gamma = a_{22} = \frac{202439144}{118586105}$,
$\tilde{a}_{32} = \frac{-121251843}{1756367063}$,
$a_{33} = \frac{12015439}{183058594}$,
$\tilde{b}_2 = \frac{1}{6}$,
$\tilde{b}_3 = \frac{2}{3}$,
$a_{31} = \frac{-6418119}{169001713}$,
$a_{21} = \frac{44004295}{24775207}$,
$\tilde{a}_{32} = \frac{-748951821}{1043823139}$,
$b_2 = \frac{1}{3}$,
$b_3 = 0$,
$\tilde{b}_1 = 1 - \tilde{b}_2 - \tilde{b}_3$,
and
$b_1 = 1 - \gamma - b_2 - b_3$.

\subsection{SSP3-ImEx(3,4,3)}
\begin{table}[H]
\centering
\begin{tabular}{c|cccc}
    $0$ & $0$ & $0$ & $0$ & $0$ \\
    $0$ & $0$ & $0$ & $0$ & $0$ \\
    $1$ & $0$ & $1$ & $0$ & $0$ \\
    $1/2$ & $0$ & $1/4$ & $1/4$ & $0$\\
    \hline
    & $0$ & $1/6$ & $1/6$ & $2/3$ \\
\end{tabular}
\quad
\begin{tabular}{c|cccc}
    $\alpha$ & $\alpha$ & $0$ & $0$ & $0$ \\
    $0$ & $-\alpha$ & $\alpha$ & $0$ & $0$ \\
    $1$ & $0$ & $1-\alpha$ & $\alpha$ & $0$ \\
    $1/2$ & $\beta$ & $\eta$ & $1/2-\beta-\eta-\alpha$ & $\alpha$ \\
    \hline
    & $0$ & $1/6$ & $1/6$ & $2/3$ \\
\end{tabular}\;,
\caption{Tableau for the 3rd-order L-stable type I ImEx-RK method: the explicit part is not FSAL, the
implicit part is not SA, hence not GSA.}
\label{tab:SSP3ImEx343}
\end{table}
where $\alpha = 0.241694260788$, $\beta = 0.0604235651970$,
and $\eta = 0.12915286960590$.

\section{Type II ImEx methods}
\label{app:Tyep_II_ImEx_Methods}
\subsection{ARS(2,2,2)}
\begin{table}[H]
    \centering
        \begin{tabular}{c|ccc}
        $0$ & $0$ & $0$ & $0$  \\
        $\gamma$ & $\gamma$ & $0$ & $0$  \\
        $1$ & $\delta$ & $1-\delta$ & $0$  \\
        \hline
        & $\delta$ & $1-\delta$ & $0$ \\
    \end{tabular}
    \quad
    \begin{tabular}{c|ccc}
        $0$ & $0$ & $0$ & $0$ \\
        $\gamma$  & $0$ & $\gamma$ & $0$ \\
        $1$ & $0$ & $1-\gamma$ & $\gamma$ \\
        \hline
        & $0$ & $1-\gamma$ & $\gamma$ \\
    \end{tabular}
    \caption{Tableau for the ARS(2,2,2) method: the explicit part is FSAL, the implicit part is  SA, hence GSA, with coefficients $\gamma = 1 - \frac{1}{\sqrt{2}}$ and $\delta = 1 - \frac{1}{2\gamma}$.}
    \label{tab:ARS222}
\end{table}

\subsection{ARS(4,4,3)}
\begin{table}[H]
    \centering
    \begin{tabular}{c|ccccc}
        $0$ & $0$ & $0$ & $0$ & $0$ & $0$  \\
        $1/2$ & $1/2$ & $0$ & $0$ & $0$ & $0$   \\
        $2/3$ & $11/18$ & $1/18$ & $0$ & $0$ & $0$   \\
        $1/2$ & $5/6$ & $-5/6$ & $1/2$ & $0$ & $0$   \\
        $1$ & $1/4$ & $7/4$ & $3/4$ & $-7/4$ & $0$    \\
        \hline
        & $1/4$ & $7/4$ & $3/4$ & $-7/4$ & $0$   \\
    \end{tabular}
    \quad
    \begin{tabular}{c|ccccc}
        $0$ & $0$ & $0$ & $0$ & $0$ & $0$ \\
        $1/2$ & $0$  & $1/2$ & $0$ & $0$ & $0$ \\
        $2/3$ & $0$  & $1/6$ & $1/2$ & $0$ & $0$ \\
        $1/2$ & $0$  & $-1/2$ & $1/2$ & $1/2$ & $0$ \\
        $1$  & $0$  & $3/2$ & $-3/2$ & $1/2$ & $1/2$  \\
        \hline
        & $0$  & $3/2$ & $-3/2$ & $1/2$ & $1/2$ \\
    \end{tabular}
    \caption{Tableau for the ARS(4,4,3) method: the explicit part is FSAL, the implicit part is SA, hence GSA.}
    \label{tab:ARS443}
\end{table}

\subsection{ARK3(2)4L[2]SA}
\begin{table}[H]
    \centering
    \renewcommand{\arraystretch}{1.5}
    \resizebox{\textwidth}{!}{%
        \scriptsize
        \begin{tabular}{l|l l l l}
            0 &  0 & 0 & 0 &  0 \\
            $\frac{1767732205903}{2027836641118}$ & $\frac{1767732205903}{2027836641118}$ & 0 & 0 & 0 \\
            $\frac{3}{5}$ &  $\frac{5535828885825}{10492691773637}$ &  $\frac{788022342437}{10882634858940}$ &  0 &  0 \\
            1 &  $\frac{6485989280629}{16251701735622}$ &  $\frac{-4246266847089}{9704473918619}$ &  $\frac{10755448449292}{10357097424841}$ &  0 \\
            \hline
              &   $\frac{1471266399579}{7840856788654}$   &        $\frac{-4482444167858}{7529755066697}$    &       $\frac{11266239266428}{11593286722821}$      &     $\frac{1767732205903}{4055673282236}$  \\
        \end{tabular}%
    } \\
    \vspace{0.2cm}
    \text{ARK3(2)4L[2]SA–ERK}  \\
    \vspace{0.2cm}
    \resizebox{\textwidth}{!}{%
        \scriptsize
        \begin{tabular}{l|l l l l}
            0 &  0 & 0 & 0 &  0 \\
            $\frac{1767732205903}{2027836641118}$ & $\frac{1767732205903}{4055673282236}$ & $\frac{1767732205903}{4055673282236}$ & 0 & 0 \\
            $\frac{3}{5}$ & $\frac{2746238789719}{10658868560708}$ & $\frac{-640167445237}{6845629431997}$ & $\frac{1767732205903}{4055673282236}$ & 0 \\
            1 & $\frac{1471266399579}{7840856788654}$ & $\frac{-4482444167858}{7529755066697}$ & $\frac{11266239266428}{11593286722821}$ & $\frac{1767732205903}{4055673282236}$\\
            \hline
            &   $\frac{1471266399579}{7840856788654}$   &        $\frac{-4482444167858}{7529755066697}$    &       $\frac{11266239266428}{11593286722821}$      &     $\frac{1767732205903}{4055673282236}$  \\
        \end{tabular}%
    } \\
    \vspace{0.2cm}
    \text{ARK3(2)4L[2]SA–ESDIRK} \\
    \caption{Tableau for the ARK3(2)4L[2]SA method: the explicit part is not FSAL, the implicit part is SA, hence not GSA.}
    \label{tab:ARK324L2SA}
\end{table}

\subsection{ARK4(3)6L[2]SA}
\begin{table}[H]
    \centering
    \renewcommand{\arraystretch}{1.5}
    \resizebox{\textwidth}{!}{%
        \scriptsize
    	\begin{tabular}{l|l l l l l l}
             0 &  0 & 0 & 0 &  0 & 0 & 0\\
             \vspace{0.15cm}
             $\frac{1}{2}$ & $\frac{1}{2}$ & 0 & 0 &  0 & 0 & 0\\
            \vspace{0.15cm}
            $\frac{83}{250}$ & $\frac{13861}{62500}$ & $\frac{6889}{62500}$ & 0 & 0 & 0 & 0 \\
            \vspace{0.15cm}
            $\frac{31}{50}$ & $\frac{-116923316275}{2393684061468}$ & $\frac{-2731218467317}{15368042101831}$ & $\frac{9408046702089}{11113171139209}$ & 0 & 0 & 0  \\
            \vspace{0.15cm}
            $\frac{17}{20}$ & $\frac{-451086348788}{2902428689909}$ & $\frac{-2682348792572}{7519795681897}$  & $\frac{12662868775082}{11960479115383}$  & $\frac{3355817975965}{11060851509271}$ & 0 & 0 \\
            \vspace{0.15cm}
            1 & $\frac{647845179188}{3216320057751}$ & $\frac{73281519250}{8382639484533}$ & $\frac{552539513391}{3454668386233}$ & $\frac{3354512671639}{8306763924573}$ & $\frac{4040}{17871}$ & 0 \\
             \hline
             \vspace{0.15cm}
              &   $\frac{82889}{524892}$   &   0     &  $\frac{15625}{83664}$     &    $\frac{69875}{102672}$  &    $\frac{-2260}{8211}$ &    $\frac{1}{4}$  \\
        \end{tabular}%
    	} \\
    \vspace{0.2cm}
    \text{ARK4(3)6L[2]SA–ERK}  \\
    \vspace{0.2cm}
    \resizebox{\textwidth}{!}{%
        \scriptsize
    	\begin{tabular}{l|l l l l l l}
             0 &  0 & 0 & 0 &  0 & 0 & 0\\
             \vspace{0.15cm}
             $\frac{1}{2}$ & $\frac{1}{4}$ & $\frac{1}{4}$ & 0 &  0 & 0 & 0\\
            \vspace{0.15cm}
            $\frac{83}{250}$ & $\frac{8611}{62500}$ & $\frac{-1743}{31250}$ & $\frac{1}{4}$  & 0 & 0 & 0 \\
            \vspace{0.15cm}
            $\frac{31}{50}$ & $\frac{5012029}{34652500}$ & $\frac{-654441}{2922500}$ & $\frac{174375}{388108}$ & $\frac{1}{4}$ & 0 & 0  \\
            \vspace{0.15cm}
            $\frac{17}{20}$ & $\frac{15267082809}{155376265600}$ & $\frac{-71443401}{120774400}$ & $\frac{730878875}{902184768}$ & $\frac{2285395}{8070912}$ & $\frac{1}{4}$ & 0 \\
            \vspace{0.15cm}
            1 & $\frac{82889}{524892}$ & 0 & $\frac{15625}{83664}$ & $\frac{69875}{102672}$ & $\frac{-2260}{8211}$ & $\frac{1}{4}$\\
            \hline
             \vspace{0.15cm}
            &   $\frac{82889}{524892}$   &   0     &  $\frac{15625}{83664}$     &    $\frac{69875}{102672}$  &    $\frac{-2260}{8211}$ &    $\frac{1}{4}$  \\
        \end{tabular}%
    	} \\
    \vspace{0.2cm}
    \text{ARK4(3)6L[2]SA–ESDIRK} \\
    \caption{Tableau for the ARK4(3)6L[2]SA method: the explicit part is not FSAL, the implicit part is SA, hence not GSA.}
    \label{tab:ARK436L2SA}
\end{table}

\end{document}